\documentclass{amsart}
\usepackage{amsthm}
\usepackage{amsmath}
\usepackage{amsfonts}
\usepackage{amssymb}
\usepackage{mathrsfs}
\usepackage{mathtools}
\usepackage{tikz-cd}
\usepackage{enumitem}
\usepackage[utf8]{inputenc}
\usepackage[left=30mm,right=30mm]{geometry}
\usepackage{hyperref}
\usepackage{caption}
\usepackage{subcaption}
\usepackage{dirtytalk} %Remove at the end

\theoremstyle{plain}% default
\newtheorem{thm}{Theorem}[section]
\newtheorem{lem}[thm]{Lemma}
\newtheorem{prop}[thm]{Proposition}
\newtheorem{cor}[thm]{Corollary}
\newtheorem*{claim}{Claim}

\newtheorem{thmintro}{Theorem}

\newtheorem{corintro}[thmintro]{Corollary}

\newtheorem{ques}{Question}

\theoremstyle{definition}
\newtheorem{defn}[thm]{Definition}
\newtheorem{rem}[thm]{Remark}

\newcommand{\C}{C}
\newcommand{\CC}{C'}
\newcommand{\M}{M_0}

\newcommand{\Cone}{C_2}
\newcommand{\Ctwo}{C_4}
\newcommand{\Cthree}{C_1}
\newcommand{\Cfive}{C_3}
\newcommand{\Cseven}{C_6}
\newcommand{\Ceight}{C_7}

\newcommand{\Cd}{C_8}
\newcommand{\Cdd}{C_9}
\newcommand{\Kd}{K_d}
\newcommand{\Cddd}{D}
\newcommand{\T}{T_0}

\newcommand{\mc}{\mathcal}

\newcommand{\eps}{\epsilon}

\newcommand{\R}{\mathbb{R}}
\newcommand{\N}{\mathbb{N}}
\newcommand{\Z}{\mathbb{Z}}

\newcommand{\mb}{\partial_*}
\newcommand{\ms}{\partial_{x_0}} % Morse stratum
\newcommand{\ov}{\overline}

\newcommand{\dH}{d_{\mathrm{Haus}}}
\newcommand{\anz}{20}

\DeclareMathOperator{\diam}{diam}

\DeclarePairedDelimiter\abs{\lvert}{\rvert}

\usepackage{xcolor}

\newcommand{\cay}{\mathrm{Cay}(G, \mc S)}
\newcommand{\pres}{\ensuremath{\langle\mc S|\mc R\rangle} }

\newlength{\mytriwidth}
\title[Sigma-compactness and stationary measures]{Sigma-compactness of Morse boundaries in Morse local-to-global groups and applications to stationary measures}

\author{Vivian He}
 \address{Department of Mathematics, University of Toronto, Toronto, ON, Canada}
 \email{vivian.he@mail.utoronto.ca}
 
\author{Davide Spriano}
 \address{Mathematical Institute, University of Oxford, Oxford, OX2 6GG, UK}
 \email{spriano@maths.ox.ac.uk}

\author{Stefanie Zbinden}
   \address{Maxwell Institute and Department of Mathematics, Heriot-Watt University, Edinburgh, UK}
	\email{sz2020@hw.ac.uk}

\begin{document}

\begin{abstract}
We show that the Morse boundary of a Morse local-to-global group is $\sigma$-compact. Moreover, we show that the converse holds for small cancellation groups. As an application, we show that the Morse boundary of a non-hyperbolic, Morse local-to-global group that has contraction does not admit a non-trivial stationary measure. In fact, we show that any stationary measure on a \emph{geodesic boundary} of such a groups needs to assign measure zero to the Morse boundary. Unlike previous results, we do not need any assumptions on the stationary measures considered.
\end{abstract}

\maketitle

\section{Introduction}
Introduced in \cite{C:Morse}, generalizing the contracting boundary of \cite{CS:contracting}, the Morse boundary is a topological space that encodes information about the \emph{Morse directions at infinity} of a metric space. In this paper we study the relation between geometric properties of a group $G$ and topological properties of its Morse boundary $\mb G$. Morse directions are geodesic rays with properties akin to those of geodesic rays in a hyperbolic spaces, and accordingly the Morse boundary is a way to define a space with similar properties to the Gromov boundary for spaces that are not hyperbolic \cite{CharneyCordesMurray:quasi-mobius, CD:stable, CH:stable_asdim,C:Morse,FioravantiKarrer:connected,KarrerMiraftabZbinden:subgroups,Liu:dynamics,M:CAT0,zalloum:symbolic,Z:free_product}. However, this comes at a price: while the Gromov boundary of a (proper) geodesic space is always compact,  the Morse boundary is compact only when the space is hyperbolic, in which case it coincides with the Gromov boundary \cite{CD:stable}. Thus, effectively one only considers non-compact Morse boundaries. We investigate the \emph{strong $\sigma$-compactness} of the Morse boundary, which is the best compactness property that a non-compact topological space might hope to have. A topological space is strongly $\sigma$-compact if it is the direct limit of countably many compact spaces. We show strong $\sigma$-compactness of the Morse boundary for a large class of examples.
\begin{thmintro}\label{thmintro:MLTG_implies_sigma_cpt}
    Let $G$ be a group satisfying the Morse local-to-global property. Then the Morse boundary of $G$ is strongly $\sigma$-compact.
\end{thmintro}
Theorem~\ref{thmintro:MLTG_implies_sigma_cpt} focuses on groups with the \emph{Morse local-to-global} property (Definition~\ref{def_morse-local-to-global}). This is a class of groups introduced by Russell, Tran, and the second author \cite{russellsprianotran:thelocal} to address the following issue: although Morse geodesic are defined to mimic the behaviour of geodesics in hyperbolic spaces, there are groups with an abundance of Morse geodesic that behave fairly different from hyperbolic groups \cite{OlshankiiOsinSapir:lacunary}.

The Morse local-to-global property makes such pathological behaviour impossible, since it controls the structure of Morse geodesics. As evidence, for a Morse local-to-global group $G$ the following hold: if $\abs{\mb G}= \infty$ then $G$ contains a non-abelian free subgroup; there are combinations theorems for stable subgroups \cite{russellsprianotran:thelocal}; the Morse geodesics form regular languages and the growth series of stable subgroups is rational \cite{CordesRussellSprianoZalloum:regularity}; if stable subgroups of $G$ are separable then $G$ is product stable separable \cite{minehspriano:separability}; and other results. 

One should expect all ``normal'' groups to satisfy the Morse local-to-global property as evidenced by the rich class of groups proven to have the Morse local-to-global property, which we list here. 
    \begin{enumerate}[start= 0]
    \item Hyperbolic groups.
    \item Groups with empty Morse boundary, for example, wide groups, which contain solvable groups, Burnside groups, uniformly amenable groups.
    \item CAT(0) groups.
    \item Hierarchically hyperbolic groups, in particular Mapping Class groups, extension of lattice Veech groups,  Artin groups of extra large type and others.
    \item Fundamental groups of closed three manifolds. \label{item:pi_1_mfld}
    \item Coarsely injective groups \cite{SZ:injective}.
    \item Convex divisible domains \cite{IslamWeisman:morse}.
    \item Groups with a bounded, consistent bicombing, this is proven in upcoming work of Cornelia Dru\c{t}u and the second and third author.\label{item:bicombing}
    \item Groups hyperbolic relative to any of the above. \label{item:rel-hyp}
    \end{enumerate}

When not specified, the reference is \cite{russellsprianotran:thelocal}. By Theorem~\ref{thmintro:MLTG_implies_sigma_cpt}, all of the above have strongly $\sigma$-compact Morse boundary. This result is new for items \eqref{item:pi_1_mfld}, \eqref{item:bicombing} and \eqref{item:rel-hyp}, in particular yielding the following. 

\begin{corintro}
    Let $G = \pi_1(M)$ be the fundamental group of a closed three manifold. Then $\mb G$ is strongly $\sigma$-compact.
\end{corintro}
This is a step forward in understanding the Morse boundary of three manifold groups, which remained fairly mysterious until recently  \cite{charneycordessisto:complete, Z:manifold}. 

The key tool in the proof of Theorem~\ref{thmintro:MLTG_implies_sigma_cpt} is the language theoretic characterizations of Morse geodesics coming from \cite{CordesRussellSprianoZalloum:regularity}. Regular languages on a given alphabet can be encoded with a finite amount of information, and thus there are countably many of them. This allows us to select countably many compact sets whose union covers the Morse boundary. 

It is natural to wonder whether the converse direction of Theorem \ref{thmintro:MLTG_implies_sigma_cpt} holds. A positive answer would be extremely interesting as it would suggest that many results about the algebra, geometry and algorithmic properties of a group could be deduced from the topology of the Morse boundary. We are able to show such a converse for small cancellation groups.
\begin{thmintro}\label{thmintro:small_cancellation}
     Let $G = \pres$ be a $C'(1/9)$ group. Then the following are equivalent
     \begin{enumerate}
         \item  $G$ is a Morse local-to-global group,
         \item $\mb G$ is $\sigma$-compact,
         \item $\mb G$ is strongly $\sigma$-compact.
         \end{enumerate}
\end{thmintro}

Theorem~\ref{thmintro:small_cancellation} constitutes the main technical result of the paper. The missing implication is $(2) \Rightarrow (1)$. The proof is in three steps. Firstly, we use the characterization of Morse geodesics in term of their \emph{intersection function} (see Definition~\ref{def:intersection-function}) developed in \cite{arzhantseva2019negative} together with a characterization of $\sigma$-compactness from \cite{Z:small_cancellation} to show a Morse local-to-global property \emph{for geodesics}. Then we turn to the quasi-geodesic case and the second step is to use the small cancellation condition to estimate the intersection function of (a perturbation of) the chosen local Morse quasi-geodesic. The final step is to use disk diagrams to approximate the chosen local quasi-geodesic with a geodesic, and use the estimate on the intersection function together with Strebel's classification of combinatorial geodesic bigons (see Lemma~\ref{lemma:strebel}) to reduce the question to a local-to-global problem for a geodesic, which we solved in the first step.

\subsection{Applications to stationary measures}
As an application to Theorem~\ref{thmintro:MLTG_implies_sigma_cpt}, we study the stationary measure of the Morse boundary of certain Morse local-to-global groups. This endeavour sits in the larger program of quantitatively understand the relation between Morse directions and the ambient space. For instance in \cite{gekhtman2014dynamics} the author shows that the critical exponent of orbits of stable subgroups in Teichm\"uller space is smaller than the one of the entire mapping class group, essentially showing that stable subgroups grow exponentially less than the whole mapping class group. A similar results is then established in \cite{CordesRussellSprianoZalloum:regularity}, for the growth rate of (infinite-index) stable subgroups in the Cayley graphs of Morse local-to-global groups. Our result shows that not only stable subgroups, but the whole Morse boundary needs to be \say{small} in the whole space. 

\begin{thmintro}\label{thmintro:MCG}
    Let $\mc{G}$ be the mapping class group of a non-sporadic surface and let $\mu$ be a probability measure on $\mc{G}$ whose support generates $\mc{G}$ as a semigroup. Let $\Delta \mc{G}$ denote either a horofunction boundary of $\mc{G}$ or the HHG boundary of $\mc{G}$. Then the Morse boundary in $\Delta \mc{G}$ has measure zero with respect to any $\mu$-stationary measure on $\Delta(\mc{G})$.
\end{thmintro}
Theorem~\ref{thmintro:MCG} implies that, in particular, the set of points of $\Delta(\mc{G})$ corresponding to axis of pseudo-Anosov elements. Theorem~\ref{thmintro:MCG} is obtained as a special case of a more general result stating that, if a group has contraction (see Definition~\ref{def:has-contraction}), the image of the Morse boundary has measure zero in any compact metric boundary that \say{encodes the geometry}. More precisely, we say that a metric space $B$ is a \emph{geodesic boundary} for $G\curvearrowright X$ if $G$ acts on $B$ by homeomorphisms and there is a $G$--equivariant map $\iota$ from the set of geodesic rays of $X$ to $B$ and so that 
\begin{itemize}
    \item for all Morse gauges $M$, the image under $\iota$ of all $M$--Morse geodesics starting at a given basepoint is compact, and,
    \item if $\gamma_1$ and $\gamma_2$ are geodesic rays representing different points in the Morse boundary, then $\iota(\gamma_1) \neq \iota (\gamma_2)$.
\end{itemize}
  A geodesic boundary for a group is a geodesic boundary for which the action $G \curvearrowright X$ is geometric. The two conditions on the geodesic rays are very natural and encode the fact that the boundary should be related to \emph{directions at infinity} and that it should not collapse the negatively-curved directions. 
\begin{thmintro}\label{thmintro:general_vanishing}
    Let $G$ be a non-hyperbolic group with the Morse local to-global property and suppose that $G$ acts geometrically on a space containing a contracting ray. Let $\mu$ be a probability measure on $G$. Then for any geodesic boundary $B$ of $G$ the image of the set of Morse rays has measure zero in $B$ with respect to any $\mu$--stationary measure. 
\end{thmintro}
The proof of Theorem~\ref{thmintro:MCG} follows the strategy of \cite{mahertiozzo:random}. We note that the class of geodesic boundaries with a stationary measure is very general. Indeed, the existence of a $\mu$--stationary measure on a boundary $B$ is guaranteed if $B$ is compact \cite[Lemma~4.3]{mahertiozzo:random}. Thus, for a finitely generated group $G$ acting geometrically on a geodesic metric space $X$, the following are compact geodesic boundaries.
\begin{enumerate}
    \item The horofunction boundary of $X$.
    \item If $G$ is an HHG, the HHG boundary of $G$.
    \item If $X$ is an ATM space (as in \cite{qingrafi:quasiredirecting}), the quasi-redirecting boundary of $X$.
\end{enumerate}
Another advantage of Theorem~\ref{thmintro:general_vanishing} is that it is very well behaved with respect to quasi-isometries. For instance, the horofunction boundary does depend on a specific space $X$ on which $G$ acts. The theorem states that if one can find a \emph{single} space on which $G$ acts geometrically that contains a contracting ray (which does not need to be periodic), then the image of Morse ray in \emph{any} horofunction boundary needs to have measure zero. 

The vanishing of Morse directions in the horofunction boundary, or more generally in convergence bordifications, was previously studied in \cite{yang:conformal}, where the author showed that a harmonic measure arising from a random walk with a finite logarithmic moment needs to be supported on the Myrberg set, which is disjoint from the set of contracting geodesic rays. Using techniques from \cite{chawlaforghanifrischtiozzo:Poisson} the same theorem can be applied without the assumption on logarithmic moment. Notably, our result does not need any moment assumption and works for all stationary measures, not only harmonic measures arising from random walks. After the first draft of this paper was written, Kunal Chawla pointed out to us new and exciting techniques coming from \cite{choi2022random} and \cite{chawlaChoiTiozzo:genericity}, giving another strategy of proof of Theorem~\ref{thmintro:general_vanishing} for measures on geodesic boundaries arising as hitting measures of random walks.

Similar results were also obtained by \cite[Theorem~1.3]{cordesdussaulegrkhtman:anembedding}, where the authors constructed a continuous map from the Morse boundary to the Martin boundary using the Ancona inequality, and as a consequence the image of the Morse boundary needs to have vanishing measure with respect to the associated harmonic measure. Our result recovers Theorem~1.3 of \cite{cordesdussaulegrkhtman:anembedding} when applied to Martin boundaries of hierarchically hyperbolic groups. 

Theorem~\ref{thmintro:general_vanishing} holds trivially if a geodesic boundary $B$ does not admit a $\mu$-stationary measure. Since the existence of a $\mu$--stationary measure on a boundary $B$ is guaranteed if $B$ is compact, we conclude the following two corollaries.  
\begin{corintro}\label{corintro:no-stationary-measure-on-morse-boundary}
    Let $G$ be a non-hyperbolic group with the Morse local to-global property and suppose that $G$ acts geometrically on a space containing a contracting ray. The Morse boundary $\mb G$ of $G$ with the metrizable topology of \cite{CM:topology} does not admit a stationary measure.
\end{corintro}
\begin{corintro}
    Let $G$ be a non-hyperbolic group with the Morse local to-global property and suppose that $G$ acts geometrically on a space containing a contracting ray. There is no topology on $\mb G$ that is compact, metrizable, where all Morse strata are closed and so that $G$ acts on $\mb G$ by homeomorphisms.
\end{corintro}

\subsection{Further questions}

\begin{ques}\label{ques:MLTG_iff_sigma_cpt}
    Let $G$ be a group. Is having Morse local-to-global property equivalent to having strongly $\sigma$-compact Morse boundary?
\end{ques}

Question~\ref{ques:MLTG_iff_sigma_cpt} is currently open, among the reason due to a severe lack of counterexamples. Indeed, all the known groups that do not satisfy the Morse local-to-global property are infinitely presented, and to the best of our knowledge there is only one class of groups that are known to have non (strongly) $\sigma$-compact Morse boundary, namely the class of small-cancellation groups described in \cite{Z:small_cancellation}. In particular, the following is open.

\begin{ques}
    Is there a finitely presented group with non $\sigma$-compact Morse boundary?
\end{ques}

A priori, there is a real difference between $\sigma$-compact and strongly $\sigma$-compact Morse boundary, as the arguments in \cite{Z:free_product} show that the Morse boundary of the free product of two groups is $\sigma$-compact if and only if both factors have strongly $\sigma$-compact Morse boundary. However, it is very reasonable to expect that for groups the two notions are the same.

\begin{ques}
    Is there a group with $\sigma$-compact, but not strongly $\sigma$-compact Morse boundary?
\end{ques}

To complete the general theme of \say{how much of the Morse local-to-global property is implied by $\sigma$--compactness of the Morse boundary?}, we ask the following. 

\begin{ques}
    Which consequences of the Morse local-to-global property follow directly from the $\sigma$-compactness of the Morse boundary?
\end{ques}

One of the main advantages of the Morse local-to-global property is that it is a quasi-isometry invariant. One might therefore hope to relate it to acylindrical hyperbolicity, since it is a famous open question whether the latter is quasi-isometry invariant. The strongest possible result in this sense one can hope is a positive answer to the following. 

\begin{ques}
    [Hard] If $G$ has $\sigma$-compact Morse boundary, does the action of $G$ on $\mb G$ satisfy strong north-south dynamics?
\end{ques}
 It is known that the action of a group on the Morse boundary satisfies weak north-south dynamics (\cite[Theorem~6.7]{Liu:dynamics} for the limit topology, \cite[Theorem~9.4]{CM:topology} for the metrizable topology). A dynamical criterion for acylindrical hyperbolicity is provided in \cite[Theorem~1.2]{sun:adynamical}, and it requires an element with strong north-south dynamics. One of the reasons why this is necessary is that acylindrical hyperbolicity implies the existence of a free subgroup and, as mentioned, there are groups with non-empty Morse boundary and no free subgroups \cite{OlshankiiOsinSapir:lacunary}. Thus, an intermediate question is the following.

\begin{ques}
    If $G$ is such that $\mb G$ is infinite and strongly $\sigma$-compact, does $G$ have a free non-abelian subgroup?
\end{ques}

\subsection*{Acknowledgements}
We are especially grateful to Abdul Zalloum for suggesting the question on stationary measures, providing the reference to \cite{mahertiozzo:random} and, generally, for many conversations on the topic on the early stages of the project. Moreover, we want to thank Kunal Chawla, Inhyeok Choi and Wenyuan Yang for being very generous with their time and answering many questions related to stationary measures and random walks. Finally, we would like to thank Matt Cordes, Marco Fraccaroli, Ursula Hamenst\"adt, Davide Perego, Alessandro Sisto, Kim Ruane, and Giulio Tiozzo for helpful conversations, clarifications and inputs on early drafts of the paper. The first named author was partially supported by NSERC CGS-D. 

\section{Preliminaries}\label{sec:prelim}
\textbf{Notation and Convention:} For the rest of this paper, $X$ denotes a proper geodesic metric space with basepoint $x_0$ and $G$ a finitely generated group acting geometrically (that is, properly and cocompactly) on $X$. For any subspace $Y\subset X$ the closest point projection from $X$ to $Y$, if it exists, is denoted by $\pi_Y: X \to 2^Y$. For points $x, y\in X$, we denote by $[x, y]$ a fixed choice of geodesic from $x$ to $y$. We call $[x,y]-\{x,y\}$ the interior of the geodesic $[x,y]$ and denote it by $(x,y)$. Likewise, we use the notation $[x,y) = [x,y] - \{y\}$ and $(x,y] = [x,y] - \{x\}$. Let $p: I \to X$ be a path, by abuse of notation, we denote its image $Im(p)$ by $p$. For $s, t\in I$, we denote the subsegment of $p$ restricted to $[s, t]$ by $p[s, t]$. Given $x, y\in p$, we denote by $[x, y]_p$ a choice of subsegment $p[s, t]$ such that $p(s) = x$ and $p(t) = y$ and $s\leq t$. If $s > t$, then $[x, y]_p$ can still be defined, we just traverse $p$ in the other direction. Note that if $p$ is a geodesic, then $s$ and $t$ are uniquely defined.

\begin{defn}[Quasi-geodesic]
    Let $Q\geq 1$ be a constant. A continuous map $\gamma : I \to X$ is a $Q$-quasi-geodesic if 
    \begin{align}
        \frac{\abs{t - s}}{Q} - Q \leq d(\gamma(s), \gamma(t))\leq Q\abs{t - s} + Q, 
    \end{align}
    for all $s, t\in I\subset \R$. 
\end{defn}

\subsection{Morse boundaries}

In this section, we recall key definitions and results about the Morse boundary and Morse geodesics. 

\begin{defn}
    A function $M: \R_{\geq 1}\to \R_{\geq 0}$ is called a \emph{Morse gauge}, if it is non-decreasing and continuous.
\end{defn}

\begin{defn}[Morseness]
    A quasi-geodesic $\gamma$ is called \emph{$M$-Morse} for some Morse gauge $M$ if every $Q$--quasi-geodesic $\lambda$ with endpoints $\gamma(s)$ and $\gamma(t)$ stays in the closed $M(Q)$--neighbourhood of $\gamma[s, t]$. A quasi-geodesic is called \emph{Morse} if it is $M$--Morse for some Morse gauge $M$.
\end{defn}

Morse geodesics satisfy a plethora of properties, which we summarize below. These properties are well known properties, variations of which can be found in \cite{C:Morse, CharneyCordesMurray:quasi-mobius, Z:manifold}.

\begin{lem}[Properties of Morse geodesics]\label{new-monster}
    Let $M$ be a Morse gauge. For each constant $Q\geq 0$ there exists a constant $D$ and a Morse gauge $M'$ such that for every $M$--Morse $Q$--quasi-geodesic $\gamma: I\to X$ and $Q$--quasi-geodesic $\gamma' : I'\to X$ the following hold.
    \begin{enumerate}[label = (\roman*)]
        \item (Subsegment) If $\gamma'$ is a subsegment of $\gamma$, then $\gamma'$ is $M$--Morse.\label{monster:subsegments}
        \item (Close to Morse quasi-geodesic) If $\gamma$ and $\gamma'$ have endpoints which are in the $Q$--neighbourhood of each other, then
        \[
        \dH(\gamma', \gamma)\leq D
        \]
        and $\gamma'$ is $M'$--Morse.\label{monster:hausdorff}
        \item {($n$--gons)} Let $(\gamma_1, \gamma_2, \ldots, \gamma_n)$ be a geodesic $n$--gon (with possibly infinite sides). If all sides $\gamma_i$ are $M$--Morse for $1\leq i \leq n-1$, then $\gamma_n$ is $M_n$--Morse, where $M_n$ only depends on $M$ and $n$. \label{monster:triangle}
        \item (Small diameter) If $\diam{(\gamma')}\leq Q$, then $\gamma'$ is $M$--Morse. \label{monster:small-diam-is-morse}
        \item (Concatenation) If $\gamma'$ is $M$--Morse and the endpoint of $\gamma$ is the same as the starting point of $\gamma'$, then $\gamma\ast\gamma'$ is $M'$--Morse. \label{monster:concatenation}
    \end{enumerate}
    
\end{lem}

We can use Morse geodesics to define the Morse boundary. The Morse boundary $\mb X$, as defined in \cite{C:Morse}, is the set of all Morse geodesic rays, where two geodesics are identified if they have bounded Hausdorff distance. A description of the topology can be found in \cite{C:Morse}. Since the Morse boundary is a quasi-isometry invariant (see \cite{C:Morse}) we can define the Morse boundary of the group $G$ as $\mb G = \mb X$. Given a Morse gauge $M$, we can define the $M$--Morse stratum of $\mb X$, denoted by $\ms^M X$, as the subset of $\mb X$ consisting of (equivalence classes of) geodesic rays starting at $x_0$ which are $M$--Morse. We list some properties of the Morse boundary and Morse strata which we will use.

\begin{lem}[\cite{C:Morse}]\label{lem:visibility}
    The Morse boundary is a visibility space.
\end{lem}

The following results from \cite{CD:stable} and \cite{C:Morse} show that Morse strata encode compact sets.
\begin{lem}[{\cite[Lemma~4.1]{CD:stable} and {\cite[Proposition 3.12]{C:Morse}}}]\label{lem:compact-sets-in-mb}
    Let $K\subset \mb X$ be compact, then there exists a Morse gauge $M$ such that $K\subset \ms^MX$. Furthermore, $\ms^MX$ is compact for all Morse gauges $M$.
\end{lem}

A topological space $Y$ is \emph{$\sigma$-compact}, if it is the union of countably many compact subsets. In light of Lemma~\ref{lem:compact-sets-in-mb} we have the following 

\begin{lem}\label{lem:def-of-sigma-compact}
    The Morse boundary $\mb X$ is $\sigma$-compact if and only if there exists an increasing sequence $(M_n)_{n\in \N}$ of Morse gauges, such that $\mb X = \cup_{n\in \N} \ms ^{M_n}X$.
\end{lem}

We call such a sequence $(M_n)_{n\in \N}$ an \emph{exhaustion} of $\mb X$. If the Morse boundary $\mb X$ is $\sigma$-compact and $(M_n)_{n\in \N}$ is an exhaustion of $\mb X$ we have that for every Morse geodesic ray $\gamma$, there exists $n = n(\gamma)$ such that $\gamma$ is $M_{n}$--Morse.

However, sometimes we need a stronger property. Namely, we want that for any Morse gauge $M$, there exists $n = n(M)$ such that all $M$--Morse geodesics are $M_n$--Morse. Below we give a formal definition of this which we call \emph{strong $\sigma$-compactness}.

\begin{defn}\label{def:strong-sigma-compact}
    The space $X$ has \emph{strongly $\sigma$-compact} Morse boundary if there exists an exhaustion $(M_n)_{n\in \N}$ of $\mb X$ such that for all Morse gauges $M$, we have that $\ms ^MX\subset \ms ^{M_n} X$ for some $n = n(M)$. We call such an exhaustion $(M_n)_{n\in \N}$ a \emph{strong exhaustion}. Equivalently, $\mb X$ is strongly $\sigma$-compact, if and only if
    \[
        \mb X = \lim_{\xrightarrow[n\in \N]{}} \ms ^{M_n} X.
    \]
\end{defn}

\subsection{The Morse local-to-global property}
In this section we define the Morse local-to-global property, a property introduced in \cite{russellsprianotran:thelocal}.

\begin{defn}\label{def:local-property}
    We say that a path $p \colon I\to X$  $L$--locally satisfies a property $(P)$ if for every $s, t \in I$ with $\abs{t - s} \leq L$ the subpath $p[s, t]$ has the property $(P)$. The quantity $L$ is called the \emph{scale}. We say that $p$ is locally $(P)$ if it is $L$--locally $(P)$ for some scale $L$.
\end{defn}

The two local properties we will consider are being Morse and being a quasi-geodesic.

\begin{defn}[Morse local-to-global]\label{def_morse-local-to-global}
    A space $X$ satisfies the \emph{Morse local-to-global} property (for short MLTG) if the following holds. For any constant $Q\geq 1$ and Morse gauge $M$ there exists a scale $L$, a constant $Q'\geq 1$ and a Morse gauge $M'$ such that every path that is $L$--locally an $M$--Morse $Q$--quasi geodesic is globally an $M'$--Morse, $Q'$--quasi-geodesic. 
\end{defn}

In \cite{russellsprianotran:thelocal}, it is shown that a path which is locally a quasi-geodesic is globally a quasi-geodesic as long as it stays close to an actual geodesic. 

\begin{lem}[{\cite[Lemma~2.14]{russellsprianotran:thelocal}}]\label{lemma:close-to-geodesic-implies-quasi-geodesic}
Let $\gamma: I \to X $ be an $L$--locally $Q$--quasi-geodesic and $C \geq 0$. Let $s, t\in I$. If $L > Q(3C+Q+2)$ and $\gamma[s,t]$ is contained in the $C$--neighbourhood of a geodesic from $\gamma(s)$ to $\gamma(t)$, then $\gamma[s, t]$ is a $Q'$--quasi-geodesic where $Q'$ depends only on $Q$ and $C$. 
\end{lem}

Although the statement of Lemma~\ref{lemma:close-to-geodesic-implies-quasi-geodesic} differs slightly from \cite[Lemma~2.14]{russellsprianotran:thelocal}, it follows from the proof of \cite{russellsprianotran:thelocal}. 

Further, we have that being close to a locally Morse quasi-geodesic implies being a locally Morse. 

\begin{lem}\label{lem:close-to-local-implies-local}
    Let $M$ be a Morse gauge and let $Q\geq 1$ be a constant. There exists a Morse gauge $M'$ such that the following holds. Let $L'\geq 0$ be a scale. There exists a scale $L$ such that any $Q$--quasi-geodesic $\gamma'$ in the $Q$--neighbourhood of an $L$--locally $Q$--quasi-geodesic $\gamma$ is $L'$--locally $M'$--Morse.
\end{lem}
\begin{proof}
    This is a direct consequence from Lemma~\ref{new-monster}\ref{monster:hausdorff} applied to subsegments of $\gamma$ and $\gamma'$. 
\end{proof}

\subsection{Strong contraction}
We now introduce strong contraction, which is slightly stronger property than being Morse.

\begin{defn}[Strongly contracting]
    Let $C\geq 0$ be a constant. We say that a geodesic $\gamma$ is \emph{$C$--contracting} if for every point $x\in X$, $\diam(\pi_\gamma(B_{d(x, \gamma)} (x)))\leq C$. A geodesic is called \emph{strongly contracting} if it is $C$--contracting for some constant $C$. 
\end{defn}

The following lemma from \cite{ACHG:contraction_morse_divergence} relates strong contraction to Morseness. 

\begin{lem}[Implication of Theorem 1.4 of \cite{ACHG:contraction_morse_divergence}]\label{lem:theorem1.4}
    Let $C\geq 0$ be a constant. There exists a Morse gauge $M$ only depending on $C$ such that any $C$--contracting quasi-geodesic is $M$--Morse.
\end{lem}

The converse of Lemma~\ref{lem:theorem1.4} does not always hold. Spaces where it does are called \emph{Morse-dichotomous} (see \cite{Z:hyperbolic}) and we have the following result.

\begin{cor}\label{morse-dich-implies-strong-sigma}
    If $X$ is Morse-dichotomous, then $\mb X$ is strongly $\sigma$-compact.
\end{cor}

Similar to Morse geodesics, strongly contracting geodesics satisfy a plethora of properties, which we summarize below.

\begin{lem}[Direct consequence of \cite{ACHG:contraction_morse_divergence} Theorem 7.1]\label{lemma:equivalence_strong_contraction1}
   For any constant $C$, there exists a constant $D$ such that any $C$--contracting geodesic $\gamma$ satisfies the \emph{$D$--bounded geodesic image property}, that is, every geodesic $\lambda$ with $d(\gamma, \lambda)\geq D$, satisfies $\diam(\pi_\gamma(\lambda))\leq D$.
\end{lem}

The following lemma states that strong contraction of geodesics behaves well under taking subsegments; a proof can be found in \cite[Theorem 1.1]{EZ:contracting}. 

\begin{lem}[Strongly-contracting subsegments]\label{lemma:subsegments}
    For all constants $C$, there exists a constant $C'$ such that the following holds. Every subgeodesic of a $C$--contracting geodesic is $C'$--contracting.
\end{lem}

To following lemma states quasi-geodesics close to strongly contracting geodesics are strongly contracting, a proof can be found for example in \cite[Lemma~1.11]{Z:free_product}.

\begin{lem}\label{lemma:hausdorff_contraction}
    There exists a function $\Phi_{d} : \R_{\geq 0}\to \R_{\geq 0}$ such that the following holds. Every $C$--quasi-geodesic segment $\lambda$ with endpoints contained in a $C$--neighbourhood of a $C$--contracting $C$--quasi-geodesic $\gamma$ is $\Phi_{d}(C)$--contracting.
\end{lem}

Strongly contracting geodesics also satisfy the following slightly technical lemmas which we will use in Section~\ref{sec:measure}

\begin{lem}\label{lem:paths-are-rectangles}
    Let $C\geq 0$ be a constant. There exists a constant $C'$ such that the following holds. Let $\gamma$ be a $C$--contracting geodesic. Let $x, y$ be points and $x'$ and $y'$ closest point projections of $x$ respectively $y$ onto $\gamma$. If $d(x', y')\geq C'$, then for all geodesics $\eta$ from $x$ to $y$ there exist $u, v$ on $\eta$ such that $d_{Haus}([u, v]_\eta, [x', y']_\gamma)\leq C'$. Moreover one can assume that $w\in [u, v]_{\eta}$ for all $w\in \eta$ which are contained in the $C$--neighbourhood of $[x', y']_\gamma$.
\end{lem}

\begin{proof}
    By Lemma~\ref{lemma:equivalence_strong_contraction1} there exists a constant $D\geq C$ only depending on $C$ such that $\gamma$ has $D$--bounded geodesic image. By Lemma~\ref{lemma:subsegments} we can assume that all subsegments of $\gamma$ have $D$--bounded geodesic image. Consequently, if $C'\geq D$ we know that $d(\eta, [x', y'])\leq D$. Let $u$ and $v$ be the first respectively last point on $\eta$ in the $(D+1)$--neighbourhood of $[x', y']_{\gamma}$ and let $u'$ and $v'$ be closest point projections of $u$ and $v$ on $[x', y']_{\gamma}$. Note that with these choices, we have $d(u, u')\leq D+1$ and $d(v, v')\leq D+1$. With this, $d([x, u]_\eta, [x', y']_{\gamma}) > D$ and since $[x', y']_{\gamma}$ has $D$--bounded geodesic image we know that $\diam(\pi_{[x', y']_{\gamma}}([x, u]_\eta))\leq D$ implying that $d(x', u')\leq D$. Hence by the triangle inequality $d(x', u)\leq 2D+1$. Similarly one can show that $d(y', v)\leq 2D+1$. Lemma~\ref{lem:theorem1.4} implies that there exists a Morse gauge $M$ only depending on $C$ such that $[x', y']_\gamma$ is $M$--Morse. Finally, Lemma~\ref{new-monster}\ref{monster:hausdorff} shows that the Hausdorff distance between $[x', y']_\gamma$ and $[u, v]_\eta$ can be bounded by a constant $C'$ only depending on $M$ and $D$ and hence only depending on $C$. Our choice of $u$ and $v$ shows that the moreover part holds.
\end{proof}

\begin{lem}\label{passing-contraction}
Let $C$ be a constant. For each constant $l$, there exists a constant $C_l$ depending only on $C$ and $l$ such that the following holds. Let $\gamma :[0, T]\to X$ be a geodesic and $S\in [0, T]$ such that $\gamma[0, S]$ is $C$--contracting. Let $\lambda: [0, T']$ be a geodesic with $d(\lambda(0), \gamma(0))\leq l$ and $d(\lambda(T'), \gamma(T))\leq l$. Then the following hold
\begin{enumerate}
    \item $\lambda[0, S]$ is $C_l$--contracting,
    \item $d_{\mathrm{Haus}}(\lambda[0, S], \gamma[0, S])\leq C_l$.
\end{enumerate}
\end{lem}
\begin{proof}
    Let $C'$ be the constant form Lemma~\ref{lem:paths-are-rectangles} applied to $\max\{C, 2l\}$. By Lemma~\ref{lemma:equivalence_strong_contraction1} there exists a constant $D$ only depending on $C$ such that $\gamma[0, S]$ has $D$--bounded geodesic image. If $T - S \leq l + D  + 1$, then (1) follows directly from Lemma~\ref{lemma:hausdorff_contraction} about geodesics with endpoints close to contracting geodesics and (2) follows from  Lemma~\ref{lem:theorem1.4} which states that contracting geodesics are Morse and Lemma~\ref{new-monster}\ref{monster:hausdorff} which states that if quasi-geodesics have close endpoints and one of them is Morse, then they have bounded Hausdorff distance. Hence we can assume that $T - S \geq l + D  + 1$. Since any geodesic of length at most $r$ is $r$--contracting, the statement follows trivially if $S\leq 2l + C' + D$ and hence we can assume that $S > 2l + C' + D$.
    With these assumptions we have by the triangle inequality that $d([\gamma(T), \lambda(T')], \gamma[0, S])\geq D+1$ and hence $\diam{(\pi_{\gamma[0, S]}([\gamma(T), \lambda(T')]))}\leq D$ since $\gamma[0, S]$ has $D$--bounded geodesic image. Let $x, y$ and $z$ be the closest point projection of $\lambda(T'), \gamma(T)$ and $\lambda(0)$ onto $\gamma[0, S]$. Note that $y = \gamma(S)$ because $\gamma$ is a geodesic. Hence the statement above implies that
    \begin{align*}
        d(x, \gamma(S))\leq d(x, y) \leq D.
    \end{align*}
    Since $d(\lambda(0), \gamma(0))\leq l$ we have that $d(z, \gamma(0))\leq2l$ by the triangle inequality. Consequently, $d(z, x)\geq S - 2l - D> C'$. By Lemma~\ref{lem:paths-are-rectangles}, there exists $u, v$ on $\lambda$ such that the Hausdorff distance between $[z, x]_{\gamma}$ and $[u, v]_{\lambda}$ is at most $C'$. The moreover part of Lemma~\ref{lem:paths-are-rectangles} further guarantees that $u = \lambda(0)$. Recall that $d(y, \gamma(S))\leq D$. By the triangle inequality, we have that $v = \lambda(t)$ for some $S - D - C' - 2l\leq t\leq S + D + C' + 2l$. Hence 
    \[
    d(\lambda(S), \gamma(S))\leq d(\gamma(S), v)+ d(v, x)\leq  2D + 2C' + 2l.
    \]
    Now (1) follows directly from Lemma~\ref{lemma:hausdorff_contraction} about geodesics with endpoints close to contracting geodesics and (2) follows from  Lemma~\ref{lem:theorem1.4} which states that contracting geodesics are Morse and Lemma~\ref{new-monster}\ref{monster:hausdorff} which states that if quasi-geodesics have close endpoints and one of them is Morse, then they have bounded Hausdorff distance.
\end{proof}

\subsection{Small cancellation}

We will subsequently define most notions about small-cancellation needed in our paper. For further background on small-cancellation we refer to \cite{lyndon1977combinatorial}. 

\textbf{Notation and Conventions.} For the rest of this section and in Section~\ref{sec:small-cancellation}, unless specified otherwise, $\mc S$ denotes a finite set of formal variables, $\mc S^{-1}$ its formal inverses and $\ov{\mc S}$ the symmetrised set $ \mc S \cup \mc S^{-1}$. A word $w$ over $\mc S$ (respectively $\ov{\mc S}$) is a finite sequence of elements in $\mc S$ (respectively $\ov{\mc S}$). By abuse of notation, we sometimes allow words to be infinite. 

Let $G = \pres$ be finitely generated group, $X = \cay$ its Cayley graph, $p$ an edge path in $X$ and $v$ a word over $\ov {\mc S}$.
\begin{itemize}
    \item By following $p$, we can read a word $w$ over $\ov {\mc S}$. We say that $p$ is labelled by $w$. We say a word $w'$ is a subword of $p$ if it is a subword of $w$.
    \item For any vertex $x\in X$ there is a unique edge path labelled by $v$ and starting at $x$.
\end{itemize}

We say a word $w$ over $\ov {\mc S}$ is \emph{cyclically reduced} if it is reduced and all its cyclic shifts are reduced. Given a set $\mc R$ of cyclically reduced words, we denote by $\overline{\mc R}$ the cyclic closure of $\mc R\cup \mc R^{-1}$. If $\mc R = \{w\}$ we sometimes denote $\ov{\mc{R}}$ by $\ov w$.

\begin{defn}[Piece]
    Let $\mc S$ be a finite set and let $\mc R$ be a set of cyclically reduced words over $\ov{\mc S}$. We say that $p$ is a piece if there exists distinct words $r, r'\in \overline{\mc R}$ such that $p$ is a prefix of both $r$ and $r'$. We say that $p$ is a piece of a word $r\in \ov{\mc R}$ if $p$ is a piece and a subword of $r$.
\end{defn}

\begin{defn}[$C'(\lambda)$ condition]
    Let $\lambda >0$ be a constant. We say that a set $\mc R$ of cyclically reduced words satisfies the $C'(\lambda)$--small-cancellation condition if for every word $r\in \overline{\mc R}$ and every piece $p$ of $r$ we have $\abs{p}<\lambda\abs{r}$. 
\end{defn}
If $\mc R$ satisfies the $C'(\lambda)$--small-cancellation condition we call the finitely generated group $G = \pres$ a $C'(\lambda)$--group. If $G = \pres$ is a $C'(\lambda)$--group, then the graph $\Gamma$ defined as the disjoint union of cycle graphs labelled by the elements of $\mc R$ is a $Gr'(\lambda)$--labelled graph as defined in \cite{gruber2018infinitely}. We can thus state and use the results of \cite{gruber2018infinitely} and \cite{arzhantseva2019negative} in the less general setting of groups satisfying the $C'(\lambda)$--small-cancellation condition.  

\begin{lem}[Lemma 2.15 of \cite{gruber2018infinitely}]\label{lem:isometric_embedding} Let $G = \pres$ be a $C'(1/6)$--group. Let $r\in \mc R$ be a relator, $\Gamma_0$ a cycle graph labelled by $r$ and let $f\colon \Gamma_0\to \cay$ be a label-preserving graph homomorphism. Then $f$ is an isometric embedding, and its image is convex.
\end{lem}

We call the image of such a label-preserving graph homomorphism the image of a relator. By abuse of notation, we sometimes use relators and their images interchangeably. The lemma has the following consequence. 

\begin{lem}\label{lemma:Morse_implies_small_relators}
    Let $M$ be a Morse gauge and let $G = \pres$ be a $C'(1/6)$--group. Then for each integer $N\geq 2$, there exists $D$ such that the following holds. If $r\in \mc R$ is a relator and $w\subset r$ is a subword such that 
    \begin{enumerate}
        \item $\abs{w}\geq \frac{\abs{r}}{N}$,
        \item and $w$ is $M$--Morse,
    \end{enumerate}
    then $\abs{r}\leq D$.
\end{lem}
\begin{proof}
    Let $w' \subset w$ be a subword with length $\abs{w'} = \frac{\abs{r}}{N}$. Then $w'$ is also $M$--Morse by Lemma~\ref{new-monster}\ref{monster:subsegments}. By Lemma \ref{lem:isometric_embedding}, the relator $r$ isometrically embeds in $G$, so the segment $r-w'$ is an $N$--quasi-geodesic. Let $x \in r-w'$ be the point in that segment with maximal distance from $w$. Since $r$ is isometrically embedded in $G$,
    \[
        d(x,w') \geq \frac{\abs{r} - \frac{\abs{r}}{N}}{2}= \frac{N-1}{2N}\abs{r}.
    \]
    On the other hand, since $w'$ is $M$--Morse, 
    \[
        d(x,w') \leq M(N).
    \]
    Combining the two inequalities we have $\abs{r}\leq D$, where
    \[
        D= \frac{2N}{N-1}M(N).
    \]
\end{proof}

\subsubsection{Disk diagrams}

\begin{defn}
    A (disk) diagram is a contractible, planar 2-complex. A disk diagram is 
    \begin{itemize}
        \item \emph{simple}, if it is homeomorphic to a disk.
        \item \emph{$\mc S$--labelled}, if all edges are labelled by an element of $\ov{\mc S}$.
        \item a diagram \emph{over $\mc R$}, if the boundary of any face is labelled by an element of $\ov{\mc R}$
    \end{itemize}
\end{defn}

Let $D$ be a disk diagram and let $\Pi$ be a face of $D$. An \emph{arc} is a maximal subpath of $D$ whose interior vertices all have degree 2. An arc is an \emph{interior arc} if its interior is contained in the interior of $D$ and an \emph{exterior arc} otherwise. Note that an exterior arc is contained in the boundary of $D$. The \emph{interior degree} of a face $\Pi$ is the number of interior arcs in its boundary and the \emph{exterior degree} of a face $\Pi$ is the number of exterior arcs in its boundary. A face is an \emph{interior face} if its exterior degree is $0$ and \emph{exterior face} otherwise.

\begin{defn}[Combinatorial geodesic bigon] A combinatorial
geodesic $n$--gon $(D, \gamma_1, \gamma_2, \ldots, \gamma_n)$ is a simple diagram $D$ whose boundary $\partial D$ is the concatenation $\gamma_1\ast \ldots \ast \gamma_n$ and such that the following conditions hold.
\begin{enumerate}
    \item Each boundary face whose exterior part is a single arc contained in one of the sides $\gamma_i$ has interior degree at least 4.\label{cond:ext}
    \item The boundary of each interior face consists of at least 7 arcs. \label{cond:int}
\end{enumerate}
\end{defn}

Combinatorial geodesic $n$--gons for $n\leq 3$ have been classified in \cite{strebel1990small} as follows.

\begin{lem}[Strebel’s classification,{\cite[Theorem 43]{strebel1990small}}]\label{lemma:strebel}
Let $D$ be a simple diagram.
\begin{enumerate}
    \item If $D$ is a combinatorial $1$-gon, then $D$ has a single face.
    \item If $D$ is a combinatorial $2$-gon (also called bigon), then either $D$ has a single face or it has shape $I_1$ from Figure~\ref{fig:strebel}.
    \item If $D$ is a combinatorial geodesic $3$--gon (also called triangle), then either $D$ has a single face or it has one of the shapes $I_2, I_3, II, III_1, IV$, or $V$ in Figure~\ref{fig:strebel}.
\end{enumerate}
\end{lem}

Each of these shapes represents an infinite family of combinatorial geodesic bigons or
triangles obtained by performing face combination at a non-degenerate, distinguished
vertex with a shape $I_1$ bigon arbitrarily many times. Figure \ref{fig:altV} shows alternate examples of each shape.

\begin{figure}[h]
\captionsetup[subfigure]{labelformat=empty}
  \centering\hfill
  \begin{subfigure}[h]{.33\textwidth}
\centering
\includegraphics[width=\mytriwidth]{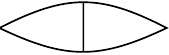}
    \caption{$\mathrm{I}_1$}
  \end{subfigure}\hfill
\begin{subfigure}[h]{.33\textwidth}\centering
\includegraphics[width=\mytriwidth]{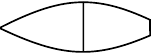}
    \caption{$\mathrm{I}_2$}
  \end{subfigure}\hfill
\begin{subfigure}[h]{.33\textwidth}\centering
\includegraphics[width=\mytriwidth]{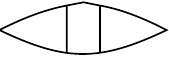}
    \caption{$\mathrm{I}_3$}
  \end{subfigure}\hfill

\begin{subfigure}[h]{.25\textwidth}\centering
\includegraphics[width=\mytriwidth]{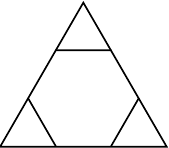}
    \caption{$\mathrm{II}$}
  \end{subfigure}\hfill
\begin{subfigure}[h]{.25\textwidth}\centering
\includegraphics[width=\mytriwidth]{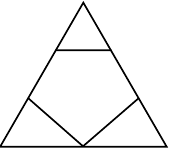}
    \caption{$\mathrm{III}_1$}
  \end{subfigure}\hfill
\begin{subfigure}[h]{.25\textwidth}\centering
\includegraphics[width=\mytriwidth]{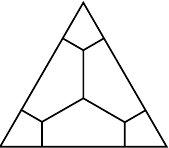}
    \caption{$\mathrm{IV}$}
  \end{subfigure}\hfill
\begin{subfigure}[h]{.25\textwidth}\centering
\includegraphics[width=\mytriwidth]{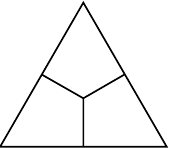}
    \caption{$\mathrm{V}$}
  \end{subfigure}\hfill

  \caption{Strebel's classification of combinatorial geodesic bigons and triangles, adopted from\cite[Figure~5]{arzhantseva2019negative}.}
  \label{fig:strebel}
\end{figure}

\begin{figure}[h]
\captionsetup[subfigure]{labelformat=empty}
  \centering\hfill
  \begin{subfigure}[h]{.33\textwidth}
\centering
\includegraphics[width=\mytriwidth]{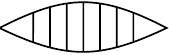}
    \caption{$\mathrm{I}_1$}
  \end{subfigure}\hfill
\begin{subfigure}[h]{.33\textwidth}\centering
\includegraphics[width=\mytriwidth]{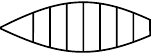}
    \caption{$\mathrm{I}_2$}
  \end{subfigure}\hfill
\begin{subfigure}[h]{.33\textwidth}\centering
\includegraphics[width=\mytriwidth]{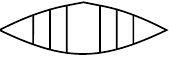}
    \caption{$\mathrm{I}_3$}
  \end{subfigure}\hfill

\begin{subfigure}[h]{.25\textwidth}\centering
\includegraphics[width=\mytriwidth]{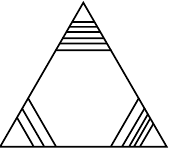}
    \caption{$\mathrm{II}$}
  \end{subfigure}\hfill
\begin{subfigure}[h]{.25\textwidth}\centering
\includegraphics[width=\mytriwidth]{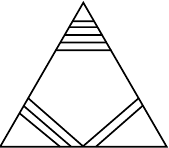}
    \caption{$\mathrm{III}_1$}
  \end{subfigure}\hfill
\begin{subfigure}[h]{.25\textwidth}\centering
\includegraphics[width=\mytriwidth]{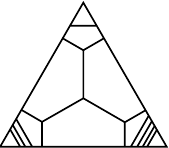}
    \caption{$\mathrm{IV}$}
  \end{subfigure}\hfill
\begin{subfigure}[h]{.25\textwidth}\centering
\includegraphics[width=\mytriwidth]{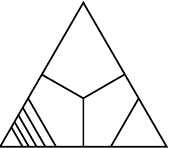}
    \caption{$\mathrm{V}$}
  \end{subfigure}\hfill

  \caption{Alternate examples of each shape, adopted from\cite[Figure~6]{arzhantseva2019negative}.}
  \label{fig:altV}
\end{figure}

\subsubsection{Morse geodesics in small-cancellation groups} In this section we highlight some consequences of \cite{arzhantseva2019negative}, which allow us to characterize Morse geodesics in small-cancellation groups in terms of their intersections with relators.

\begin{defn}[Intersection function]\label{def:intersection-function}
Let $G = \pres$ be a $C'(1/6)$--group. Let $\gamma$ be an edge path in $\cay$. The intersection function of $\gamma$ is the function $\rho : \N\to \R_+$ defined by 
\begin{align*}
    \rho(t) = \max_{\substack{\abs{r}\leq t\\r\in\ov{\mc R}}}\left\{\abs{w}\Big\vert \text{$w$ is a subword of $r$ and $\gamma$}\right\}.
\end{align*}
\end{defn}
\begin{rem}
    In light of Lemma \ref{lem:isometric_embedding}, the intersection function $\rho$ of a geodesic $\gamma$ is equal to the function $\rho'$, defined via
    \begin{align*}
    \rho'(t) = \max_{\abs{\Gamma_0}\leq t}\left\{\abs{\Gamma_0 \cap \gamma}\right\},
    \end{align*}
    where $\Gamma_0$ ranges over all images of relators.
\end{rem}

The following lemma is a consequence of \cite[Theorem 1.4]{ACHG:contraction_morse_divergence} and \cite[Corollary 4.14, Theorem 4.1]{arzhantseva2019negative}. It shows the relation between contraction, Morseness and the intersection function. In \cite{Z:small_cancellation} it is explained in more detail how Lemma~\ref{lem:4.14} follows from \cite[Theorem 4.1, Corollary 4.14]{arzhantseva2019negative}

\begin{lem}[Consequence of {\cite[Theorem 4.1, Corollary 4.14]{arzhantseva2019negative}}]\label{lem:4.14} Let $G = \pres$ be a $C'(1/6)$--group. Let $\alpha$ be a geodesic in $X =\cay$ and let $\rho$ be its intersection function. Then,
\begin{enumerate}[label=(\roman*)]
    \item The geodesic $\alpha$ is Morse if and only if $\rho$ is sublinear.
    \item If $\rho$ is sublinear, $\alpha$ is $M$-Morse for some Morse gauge $M$ only depending on $\rho$.
    \item If $\alpha$ is $M$-Morse, then $\rho\leq \rho'$ for some sublinear function $\rho'$ only depending on $M$. 
\end{enumerate}
\end{lem}

\subsubsection{Increasing partial small-cancellation condition}

In \cite{Z:small_cancellation}, the notion of $C'(1/f)$ for ``viable'' functions $f$ is introduced, which is used to define the increasing partial small-cancellation condition (see Definition~\ref{def:IPSC}).

\begin{defn}\label{def:viable}
    We say that a non-decreasing function $f: \N \to \R_+$ is viable if $f(n)\geq 6$ for all $n$ and $\lim_{n\to \infty} f(n) = \infty$. 
\end{defn}

Observe that a consequence of a function $f$ being viable is that the function $\rho(n) = n/f(n)$ (or equivalently the function $\rho'(n) =\max\{\rho'(n-1), n/f(n)\}$) is sublinear. 

\begin{defn}\label{def:c1f} Let $f$ be a viable function. We say that a finitely generated group $G = \pres$ is a $C'(1/f)$--group if $\mc R$ is infinite and for every piece $p$ of a relator $r\in \ov{\mc R}$ we have that $\abs{p}<\abs{r} / f(\abs{r})$. 
\end{defn}

In other words, in $C'(1/f)$--groups, pieces of large relators have to be smaller fractions of their relators than pieces of smaller relators. Requiring that viable functions satisfy $f(n)\geq 6$ ensures that all $C'(1/f)$--groups are $C'(1/6)$--groups. Instead of requiring that a set of relators $\mc R$ satisfies the $C'(1/f)$ condition as a whole, we can also require that certain words satisfy the $C'(1/f)$ condition with respec to $\mc R$.

\begin{defn}\label{def:loacl:c1f}
    Let $f$ be a viable function, let $x$ be a reduced word and let $\mc R$ be a set of cyclically reduced words. We say that the pair $(x, \mc R)$ satisfies the $C'(1/f)$--small-cancellation condition if every common subword $p$ of $x$ and a relator $r\in \ov {\mc R}$ satisfies $\abs{p}< \abs{r}/f(\abs{r})$.
\end{defn}

The following definition introduced in \cite{Z:small_cancellation} quantifies having sufficiently long subwords $w$ of longer and longer relators such that $(w, \mc R)$ satisfy the $C'(1/f)$--small-cancellation condition.

\begin{defn}[IPSC]\label{def:IPSC}
     Let $G = \pres$ be a finitely generated group. We say that $G$ satisfies the \emph{increasing partial small-cancellation condition (IPSC)} if for every sequence $(n_i)_{i\in \N}$ of positive integers, there exists a viable function $f$ such that the following holds. For all $K\geq 0$ there exists $i\geq K$ and a relator $r = xy\in \ov{\mc R}$ satisfying: 
    \begin{enumerate}[label= \roman*)]
        \item $\abs{r}\geq n_i$,
        \item $\abs{x}\geq \abs{r}/i$,
        \item the pair $(x, \mc R)$ satisfies the $C'(1/f)$--small-cancellation condition.
    \end{enumerate}
\end{defn}

We care about groups satisfying the IPSC condition due to the following result.

\begin{lem}[{\cite[Theorem~B]{Z:small_cancellation}}]\label{prop:non-sigma-compact-condition}
    Let $G = \pres$ be a $C'(1/6)$--group. The Morse boundary of $G$ is non-$\sigma$-compact if and only if $G$ satisfies the IPSC condition.
\end{lem}

\section{The Morse-local-to-global property in small-cancellation groups}\label{sec:small-cancellation}

In this Section we prove the implication $(2)\implies (1)$ of Theorem~\ref{thmintro:small_cancellation}. That is, we show that if a $C'(1/9)$--group has a $\sigma$-compact Morse boundary, then it satisfies the Morse local-to-global property. We first show that if a $C'(1/6)$--group $G = \pres$ has $\sigma$-compact Morse boundary, then any geodesic which is $L$--locally Morse for a sufficiently large scale $L$ is globally Morse (see Definition~\ref{def:geodesicMLTG}). The proof of this relies on Lemma~\ref{prop:non-sigma-compact-condition}, which states that $C'(1/6)$--groups have $\sigma$-compact Morse boundary if and only if they satisfy the IPSC condition. As a second step, we approximate a locally Morse quasi-geodesic by a path $p$, which we show to have ``small'' intersection function by using disk diagram techniques. Lastly, we show using combinatorial geodesic bigons, that for all $s, t$, we have that $p[s, t]$ stays close to the geodesic $[p(s), p(t)]$. The last step is the only place where we require that $G$ is a $C'(1/9)$--group instead of a $C'(1/6)$--group. If that step is shown for $C'(1/6)$--groups, then the whole result can be upgraded to $C'(1/6)$--groups.

\begin{defn}[geodesic MLTG]\label{def:geodesicMLTG}
We say that a geodesic metric space $X$ satisfies the \emph{geodesic MLTG} property if for every Morse gauge $M$, there exists a constant $L$ and a Morse gauge $M'$ such that every geodesic which is $L$--locally $M$--Morse is $M'$--Morse.
\end{defn}

Further we say that a $C'(1/6)$ group $G = \pres$ satisfies the geodesic MLTG property if $X = \cay$ satisfies the geodesic MLTG property.

\begin{prop}\label{lem:sigma_cpct_implies_geosdesicMLTG}
	Let $G = \pres$ be a $C'(1/6)$--group. If $\mb G$ is $\sigma$-compact, then $X = \cay$ satisfies the geodesic MLTG property. 
\end{prop}
\begin{proof}

By Lemma~\ref{prop:non-sigma-compact-condition}, which states that for $C'(1/6)$--groups having $\sigma$-compact Morse boundary is equivalent to not satisfying the IPSC condition it suffices to show that if $G$ does not have the geodesic MLTG property, then $G$ satisfies the IPSC condition. 

Assume that $G = \pres$ does not have the geodesic MLTG property. Consequently, there exists a Morse gauge $M$ such that for all Morse gauges $M'$ and scales $L$ there exists a geodesic $\gamma_{L,M'}$ which is $L$--locally $M$--Morse but not $M'$--Morse. 

Thanks to Lemma~\ref{lem:4.14} we know that in a $C'(1/6)$ group a  geodesic $\gamma$ is $M$--Morse if and only if its intersection function $\rho_\gamma$ is sublinear. Further, the Morse gauge $M$ and the intersection function $\rho_\gamma$ can be bounded in terms of each other. Hence not satisfying the geodesic MLTG property is equivalent to the following.

There exists a sublinear function $\rho$ such that for all sublinear functions $\rho'$ and for all scales $L$, there exists a geodesic $\gamma_{L, \rho'}$ which is $L$-locally $\rho$--intersecting but not $\rho'$--intersecting.

If $\gamma_{L, \rho'}$ is not $\rho'$--intersecting, then there has to exist a relator $r = xy$ such that $x$ is a subpath of $\gamma_{L, \rho'}$ and $\abs{x}>\rho'(\abs{r})$. In other words, there exists a subpath $x$ of a relator $r$ such that $x$ is $L$--locally $\rho$--intersecting but $\abs{x} > \rho'(\abs{r})$.

Now we can start proving that $G$ satisfies the IPSC. Let $(n_i)_{i\in \N}$ be a sequence of integers. We want to show that there exists a viable function $f$ such that for each $K$ there exists $i \geq K$ and a relator $r = xy$ such that the pair $(i, r=xy)$ satisfies conditions $i)-iii)$ from Definition~\ref{def:IPSC}. Since increasing $(n_i)$ makes it harder to satisfy these conditions, it suffices to consider sequences $(n_i)_{i\in \N}$ which satisfy $n_{i+1}\geq n_i+1$ and $n_i \geq i^2$ for all $i$.

In fact, we will construct a viable function $f$, a sequence of relators $r_i = x_iy_i$ and of integers $k_i\geq k_{i-1}+1$ such that for all $i$,
\begin{enumerate}[label = \Roman*)]
    \item $\abs{x_i}\geq \abs{r_i}/k_i$,\label{1}
    \item $\abs{r_i}\geq n_{k_i}$,\label{2}
    \item $(x_i, \mc R)$ satisfies the $C'(1/f)$--small-cancellation condition.\label{3}
\end{enumerate}

The existence of such a sequence of relators will conclude the proof.

Let $\rho'$ be the following function

\begin{align*}
    \rho'(t) = \max\left\{\frac{t}{j} ,\rho (t), \rho'(t-1)\right\} \quad \text{for $n_j\leq t <n_{j+1}$.}
\end{align*}
Note that $\rho'$ is sublinear and non-decreasing. 

We start by inductively constructing the sequence of relators $r_i = x_iy_i$ and numbers $k_i$ satisfying conditions \ref{1} and \ref{2}. We also need the additional quantities $m_i= \vert r_i \vert / \vert x_i\vert$ which we will ensure satisfy $m_i \geq k_i/2$,  and $L_i$, the scales we use need to find geodesics $\gamma_{L_i, \rho'}$ as in the hypothesis. 

To start the induction, set  $k_{0} = 6$,  $m_0 = 3$ and let $x_0$ be the empty word, so that $\vert x_0 \vert= 0$. For $i \geq 1$ define

\[
L_i = n_{\max\{\vert x_{i-1} \vert, {2\lceil{m_{i-1}}\rceil+2}\}}.
\]

Let $r_i = x_i'y_i'$ be the smallest relator such that
\begin{enumerate}
    \item $x_i'$ is $L_i$--locally $\rho$--intersecting.
    \item $x_i'$ is not $\rho'$--intersecting.
\end{enumerate}
Note that such a relator exists as argued above. 
Since $x_i'$ is not $\rho'$--intersecting and $\rho'\geq \rho$, we know that $\abs{x_i'} > L_i$. Define $k_i$ as the integer such that $n_{k_i}\leq \abs{r_i}< n_{k_{i}+1}$. Since $\abs{x_i'} > L_i$ we have that 
\[
    \abs{r_i} \geq \abs{x_i'} > L_i = n_{\max\{\vert x_{i-1} \vert, {2\lceil{m_{i-1}}\rceil+2}\}},
\]
and hence
\begin{equation}\label{eqn:lowerbound_k_i}
    k_i\geq\max\{\vert x_{i-1} \vert, {2\lceil{m_{i-1}}\rceil+2}\}.
\end{equation}

Since $x_i'$ is not $\rho'$--intersecting, we know that $\abs{x_i'} > \rho'(\abs{r_i}) \geq \abs{r_i}/k_i$. By choosing $x_i$ as an appropriate subsegment of $x_i'$, we can ensure that 
\begin{equation}\label{eqn:bound_x_i}
    \frac{\abs{r_i}}{k_i} < \abs{x_i} \leq 2 \frac{\abs{r_i}}{k_i}.
\end{equation}

Define $m_i = \abs{r_i}/\abs{x_i}$. Inequalities \eqref{eqn:lowerbound_k_i} and \eqref{eqn:bound_x_i} show that 
\begin{align}\label{eqmiki}
\lceil m_{i-1}\rceil +1\leq\frac{k_i}{2}\leq m_i\leq k_i.
\end{align}
In particular,
\begin{align}\label{eq:mis}
    m_{i-1}&\leq m_i,\\
    k_{i-1}+1&\leq k_i,
\end{align}
where we used the left most inequality of \eqref{eqmiki} and that by induction $k_{i-1}/2\leq m_{i-1}$.
Further, since $n_j\geq j^2$ for all $j$, we have that $\abs{x_i}> \abs{r_i}/k_i\geq k_i \geq \abs{x_{i-1}}$. In particular,
\begin{align}\label{increasingxis}
    \abs{x_{i-1}}\leq \abs{x_i}.
\end{align}

With this, \ref{1} and \ref{2} are satisfied for all $i$. It remains to construct a viable function $f$ such that \ref{3} is satisfied for all $i$. We will first construct a sublinear function $g(t)$.

To do so, we define the following functions for all $i$.

\begin{align*}
        g_i(t) = 1+ \begin{cases}
            t &\text{if $t\leq \abs{r_i}$}\\
            \frac{t}{m_i} &\text{otherwise.}
        \end{cases}
    \end{align*}

    Define 
    \[
    g_i'(t) = \max\{\rho'(t), g_i(t)\} \qquad \text{and} \qquad  g(t) = \min_{i \in \N}\{g_i'(t)\} = \max\{\rho'(t), \min_{i\in \N}\{g_i(t)\}\}.
    \]
    Lemma~\ref{lem:construct-sublinear-functions} shows that $\min_{i\in \N}\{g_i(t)\}$, and hence $g$, is sublinear. We want to show that $(x_i, \mc R)$ satisfies $\abs{p}< g(\abs{r})$ for all relators $r\in \mc R$ and pieces $p$ which are a subword of $x_i$.

    Fix $i$, let $r\in \mc R$ be a relator and let $p\subset r$ be a piece which is a subword of $x_i$. We want to show that $\abs{p}< g_j'(\abs{r})$ for all $j$. 
    \begin{itemize}
        \item \textbf{Case 1: $i< j$.} 
        \begin{itemize}
            \item  \textbf{Case 1.1: $\abs{r}\geq \abs{r_j}$.} In this case 
                \begin{align*}
                    \abs{p}\leq \abs{x_i}\leq \abs{x_j} = \frac{\abs{r_j}}{m_j}\leq \frac{\abs{r}}{m_j} < g_j(\abs{r}) \leq g_j'(\abs{r}).
                \end{align*}
                For the second step, we used \eqref{increasingxis}.
            \item  \textbf{Case 1.2: $\abs{r} < \abs{r_j}$.} By the definition of $g_j$, we have that $g_j(\abs{r}) = \abs{r}+1$ and hence 
            \[
            \abs{p} < \abs{r}+1 = g_j(\abs{r}) \leq g_j'(\abs{r}).
            \]
        \end{itemize}
        \item \textbf{Case 2 : $i\geq j$.}
        \begin{itemize}
            \item \textbf{Case 2.1: $\abs{r}\geq \abs{r_i}$.} In this case 
            \begin{align*}
                \abs{p}\leq \abs{x_i} = \frac{\abs{r_i}}{m_i}\leq \frac{\abs{r}}{m_i}\leq \frac{\abs{r}}{m_j} < g_j(\abs{r}) \leq g_j'(\abs{r}).
            \end{align*}
            The fourth inequality is due to \eqref{eq:mis}, which implies that $m_j\leq m_i$.
            \item  \textbf{Case 2.2: $\abs{r} < \abs{r_i}$.} In the definition of $r_i$ we chose the relator of minimal length which has a subword which is $L_i$--locally $\rho$--intersecting but not $\rho'$--intersecting. Since $p$ is a subword of $x_i$ it is $L_i$--locally $\rho$--intersecting. Since $\abs{r}< \abs{r_i}$, minimality of $\abs{r_i}$ implies that $p$ is $\rho'$--intersecting. Hence
            \[
                \abs{p}\leq \rho'(\abs{r}) < g_j'(\abs{r}).
            \] 
        \end{itemize}
    \end{itemize}

    Lastly, we use $g$ to construct a viable function $f$ such that $(x_i, \mc R)$ satisfies the $C'(1/f)$--condition. Neither $f$ nor $g$ will depend on $i$. 
    
    Define $f'(n) = n/g(n)$. Since $f$ is sublinear, the function $f'$ tends to infinity. Define the function $f''$ as 
    \begin{align*}
        f''(n) = \min_{n\leq k} \{f'(k)\}.
    \end{align*}
    Since $f$ tends to infinity, so does $f''$. Moreover $f''$ is non-decreasing. We have that $f''\leq f'$ so being $C'(1/f')$ implies being $C'(1/f'')$. Lastly define the function $f$ as follows
    \begin{align*}
        f(n) = \max\{6, f''(n)\}.
    \end{align*}
    Note that $f$ is viable. Since $G$ is a $C'(1/6)$--group, being $C'(1/f'')$ implies being $C'(1/f)$. Let $i\geq 1$. We have shown above that for any relator $r\in \mc R$ and piece $p\subset r$ which is a subword of $x_i$ we have that $\abs{p}<g(\abs{r}) = \abs{r}/f'(\abs{r})$. Since $f''\leq f'$ we have that $\abs{p}<\abs{r}/f''(\abs{r})$ and since $G$ satisfies the $C'(1/6)$--small-cancellation condition, we have that $\abs{p} < \abs{r}/f(\abs{r})$. Hence we have successfully constructed a viable function $f$ such that $(x_i, \mc R)$ satisfies the $C'(1/f)$--small-cancellation, which concludes the proof.
\end{proof}

\begin{lem}\label{lem:construct-sublinear-functions}
    Let $(m_i)_{i\in \N}$ and $(l_i)_{i\in \N}$ be increasing sequences of integers going to infinity. For each $i\in \N$, define $g_i$ as follows
    \begin{align*}
        g_i(t) = \begin{cases}
            t &\text{if $t\leq l_i$,}\\
            \frac{t}{m_i} &\text{otherwise.}
        \end{cases}
    \end{align*}
    Let $g = \min \{ g_i\}$. Then $g$ is sublinear.
\end{lem}
\begin{proof} First note that $g$ is well defined. Indeed, the sequence $(l_i)$ is increasing and diverges, hence for each $t$ there are only finitely many values of $i$ such that $g_i(t) \neq t$, implying that the minimum is well defined.  

To show that $g$ is sublinear, it is enough to show that for each $N$, there exists and integer $T$ such that for all $t\geq T$ we have that $g(t)< \frac{t}{N}$. Fix $N$, since $(m_i)_{i\in \N}$ goes to infinity, there exists $i$ such that $m_i> N$ define $T = l_i$. For $t\geq T$ we have that $g_i(t) = t/m_i <t/N$ and since $g$ is the minimum of all $g_i$ we have that $g(t)<t/m_i$.
\end{proof}

The following lemma shows that Lemma~\ref{lemma:Morse_implies_small_relators} also works for subsets which are locally Morse, as long as the groups satisfies the geodesic MLTG property.

\begin{lem}\label{lemma:geodesicMLTG_implies_small_relators}
    Let $M$ be a Morse gauge and let $G = \pres$ be a $C'(1/6)$--group which satisfies the geodesic MLTG property. Then there exists a scale $L$ such that for each integer $N\geq 2$, there exists $D$ such that the following holds. If $r\in \mc R$ is a relator and $w\subset r$ is a subword such that 
    \begin{enumerate}
        \item $\abs{w}\geq \frac{\abs{r}}{N}$,
        \item and $w$ is $L$--locally $M$--Morse,
    \end{enumerate}
    then $\abs{r}\leq D$.
\end{lem}
\begin{proof}
    Since $G$ satisfies the geodesic MLTG, there exist a Morse gauge $M'$ and constant $L$ such that any $L$--locally $M$--Morse geodesic is $M'$--Morse. Let $w' \subset w$ be a subword with length $\abs{w'} = \frac{\abs{r}}{N}$. In particular, $\abs{w'} \leq \frac{\abs{r}}{2}$, so it is a geodesics segment.  Since $w$ is $L$--locally $M$--Morse, so is $w'$. It follows that $w'$ is $M'$--Morse. By Lemma~\ref{lemma:Morse_implies_small_relators}, there exists $D$ such that $\abs{r}\leq D$.
\end{proof}

\subsection{Construction of auxiliary paths}

In this section we assume that $G = \pres$ is a $C'(1/9)$--small-cancellation group with $\sigma$-compact Morse boundary. We will show that, for an appropriate scale $L$, $L$--locally Morse quasi-geodesics are Morse quasi-geodesics. 

Let $Q\geq 1$ be a constant and let $M$ be a Morse gauge. Let $\gamma$ be a $2L$--locally $M$--Morse $Q$--quasi-geodesic for a constant $L$ to be determined later. For each $i \in \Z$ let $\eta_i$ be a geodesic from $\gamma(iL) = a_i$ to $\gamma((i+1)L)=a_{i+1}$.

\begin{lem}\label{lem:eta_i-are-morse}
    There exists a Morse gauge $\M$ not depending on $\gamma$ such that $\eta_i$ is $\M$--Morse for all $i$.
\end{lem}
\begin{proof}
    The geodesic $\eta_i$ has endpoints on an $M$--Morse geodesic, so Lemma~\ref{new-monster}\ref{monster:hausdorff} implies that $\eta_i$ is $\M$--Morse for some $\M$ only depending on $M$.
\end{proof}

Let $\Gamma$ be an image of some relator $r\in \mc R$ and let $W$ be a subpath of $\Gamma$ with $\abs{W}\leq \abs{\Gamma}/2$ and which intersects both $\eta_i$ and $\eta_{i+1}$. We define $x_{W}$ and $y_{W}$ to be the points on $\eta_i$ and $\eta_{i+1}$ which are contained in $W$ and minimize $d(a_i, x_W)$ and $d(a_i, y_W)$ respectively. Over all such $W$, let $W_i$ be the one that maximizes $d(a_i, x_{W_i}) +d(a_i, y_{W_{i}})$. We define $x_i = x_{W_i}$ and $y_i = y_{W_i}$. We will first show that $d(a_i, x_i)$ and $d(a_i, y_i)$ is small.

\begin{lem}\label{lem:small_paths}
    There exists a constant $\C$ not depending on $\gamma$ and $L$ such that $d(a_i, x_{i})\leq \C$ and $d(a_i, y_{i})\leq \C$ for all $i$.
\end{lem}

\begin{proof}
    Let $b_i$ be the last point on $\eta_{i-1}$ which is also contained in $\eta_i$. By the definition of $x_i$ and $y_i$, the point $b_i$ lies on $[x_i, a_i]_{\eta_{i-1}}$ and $[a_i, y_{i}]_{\eta_i}$. 
    
    Consider the triangle $\Delta$ with sides $[x_i, y_i], [y_i, b_i]_{\eta_i}$ and $[b_i, x_i]_{\eta_{i-1}}$. By construction, the three sides of $\Delta$ only intersect at their respective endpoints. Let $D$ be a minimal disk diagram with $\partial D = \Delta$. Since $\Delta$ is  a simple path, we know that $D$ is a simple disk diagram.
    
    Since in addition each side of $\Delta$ is a geodesic, we can use the classification of geodesic triangles to better understand how $D$ looks like. 

    Assume there is a face $\Pi$ of $D$ which has exactly one exterior arc $\alpha = \alpha_1\cup \alpha_2$ for which $\alpha_1$ is contained in $[b_i, x_i]_{\eta_{i-1}}$ and $\alpha_2$ is contained in $[x_i, y_i)$. We show this cannot happen. The classification of geodesic triangles implies that $\partial \Pi = \alpha \ast \beta$ or  $\partial \Pi = \alpha \ast \beta_1\ast \beta_2$ for some arc $\beta$ respectively arcs $\beta_1$ and $\beta_2$ which are pieces. Since  $[b_i, x_i]_{\eta_{i-1}}$ is a geodesic we have that $\abs{\alpha_1}\leq \abs{\Pi}/2$. Since relators are convex, we know that $[x_i, y_i]$ is a subpath of the image of a relator and hence $\alpha_2$ is a piece, implying that $\abs{\alpha_2}\leq \abs{\Pi}/6$. Consequently, $\abs{\beta}\geq \abs{\Pi}/3$, respectively $\abs{\beta_1}+\abs{\beta_2}\geq\abs{\Pi}/3$, a contradiction to $\beta$, respectively $\beta_1$ and $\beta_2$, being pieces. 
    
\begin{figure}
    \centering
    \includegraphics[width=1\linewidth]{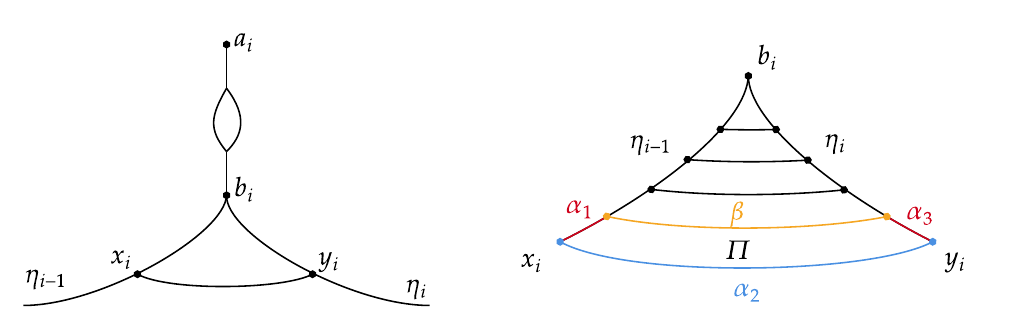}
    \caption{By classification of geodesic triangles, the disk diagram $D$ must look like the right figure.}
    \label{fig:disk}
\end{figure}

    Lemma~\ref{lemma:strebel} about the classification of combinatorial geodesic triangles shows that $D$ has to be as in Figure~\ref{fig:disk}. In particular, there exists a face $\Pi$ with $\partial \Pi = \alpha_1\ast \alpha_2\ast \alpha_3\ast \beta$ where $\beta$ is a piece, $\alpha_1$ is contained in $[x_i, b_i]_{\eta_{i-1}}$, $\alpha_2 = [x_i, y_i]$ and $\alpha_3$ is contained in $[b_i, y_i]_{\eta_{i}}$. We have that $\abs{\alpha_2}\leq \abs{\Pi}/2$ since $[x_i, y_i]$ is a geodesic and $\abs{\beta}<\abs{\Pi}/6$ since $\beta$ is a piece. Thus, for at least one of $j =1, 3$, we have that $\abs{\alpha_j}\geq \abs{\Pi}/6$, say this holds for $j=1$. By Lemma~\ref{lem:eta_i-are-morse}, we have that $\eta_{i-1}$, and hence $\alpha_1$, is $\M$--Morse, where $\M$ does not depend on $\gamma$. Now we can apply Lemma~\ref{lemma:geodesicMLTG_implies_small_relators} to $\alpha_1$ to get that $\abs{\Pi}\leq \C_r$ for some constant $\C_r$ not depending on $\gamma$. Since both $x_i$ and $y_i$ lie on $\Pi$, we have that $d(x_i, y_i)\leq \C_r$. 
    
    Since $\gamma$ is $2L$--locally an $M$--Morse $Q$--quasi-geodesic and the geodesics $\eta_i$ have endpoints on $\gamma$, there exist $(i-1)L\leq s \leq iL \leq t\leq (i+1)L$ such that $d(\gamma(s), x_i)\leq M(1)$ and $d(\gamma(t), y_i)\leq M(1)$. Since $\gamma$ is $2L$--locally a $Q$--quasi-geodesic, we can bound $s-t$ in terms of $Q, M(1)$ and $\C_r$. In turn, we can bound $d(\gamma(s), \gamma(iL))$ and $d(\gamma(t), \gamma(iL))$ in terms of $Q, M(1)$ and $\C_r$. Hence, using the triangle inequality one more time, we can see that $d(x_i, a_i)\leq\C$ and $d(y_i, a_i)\leq \C$ for a constant $\C$ only depending on $Q, M(1)$ and $\C_r$, which do not depend on $\gamma$.
\end{proof}

\begin{cor}\label{cor:no_overlap}
    Let $\Cthree$ be a constant. For large enough $L$ we have that 
    \[
    d(a_i, y_{i}) + \Cthree < d(a_{i}, x_{i+1}),
    \]
    for all $i$.
\end{cor}
\begin{proof} We have that $d(a_i, a_{i+1})\geq \eps (L)$ for a term $\eps(L)$ depending linearly on $L$. Hence for large enough $L$, we can ensure that $d(a_{i}, a_{i+1})\geq \Cthree+2\C$, which implies the desired result. 
\end{proof}

We now construct the auxiliary path. let 
\begin{align*}
    p = [\gamma(0),x_0]_{\eta_0} \ast [x_0, y_0] 
    \ast [y_0, x_1]_{\eta_1} \ast [x_1, y_1] 
    \ast [y_1,x_2]_{\eta_2} \ast [x_3,y_3] \ldots,
\end{align*}
which is depicted in Figure~\ref{fig:aux_path}. This path is well-defined by Corollary \ref{cor:no_overlap}. 
\begin{figure}
    \centering
    \includegraphics[width=1\linewidth]{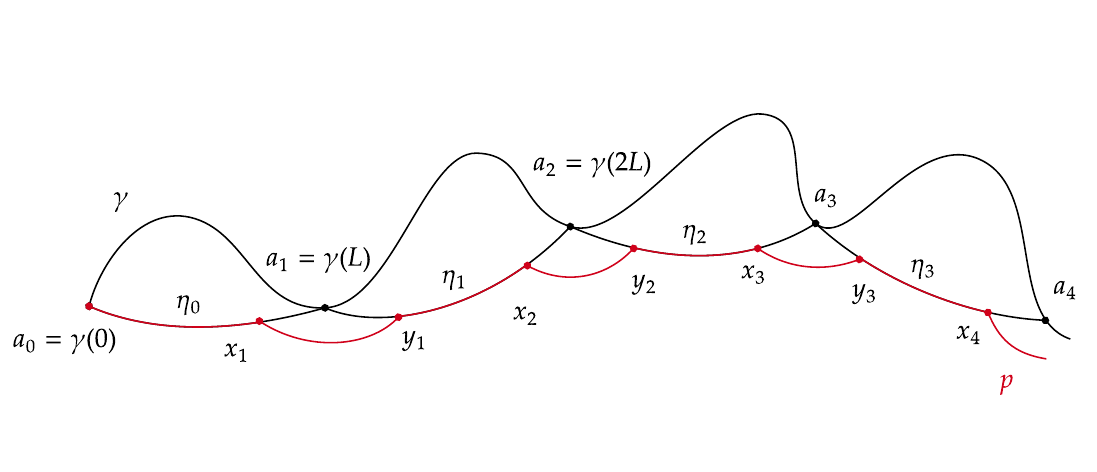}
    \caption{The auxiliary path constructed from $\gamma$. For $L$ large enough, the segments $[x_i,y_i]$ do not overlap by Corollary \ref{cor:no_overlap}.} 
    \label{fig:aux_path}
\end{figure}

\begin{lem}\label{lem:close-to-gamma}
    There exists a constant $\Cone$ not depending on $\gamma$ such that $d_{\mathrm{Haus}}(p, \gamma)\leq \Cone$.
\end{lem}
\begin{proof} 
    Since $\gamma$ is $L$--locally $M$--Morse, the segment $\gamma[iL, (i+1)L]$ is $M$--Morse. By Lemma~\ref{new-monster}\ref{monster:hausdorff}
    \[
        d_{\mathrm{Haus}}(\eta_i, \gamma[iL, (i+1)L])\leq \Cone'.
    \]
    for a constant $\Cone'$ not depending on $\gamma$. Lemma~\ref{lem:small_paths} concludes the proof.
\end{proof}

\begin{lem}\label{lem:p-is-locally-morse}
    There exists a Morse gauge $M'$ not depending on $\gamma$ such that for every scale $L'$ the following holds. For $L$ large enough, $p$ is $L'$--locally $M'$--Morse. 
\end{lem}
 \begin{proof} 
 By Lemma~\ref{new-monster}\ref{monster:concatenation} and Lemma~\ref{new-monster}\ref{monster:small-diam-is-morse}, there exists a Morse gauge $M'$ only depending on $\M$ and $\C$ such that the concatenation of two $\M$--Morse geodesics and a geodesic of length at most $2\C$ is $M'$--Morse. If we choose $L$ large enough such that $d(y_i, x_{i+1})\geq L'$, which we can by Lemma~\ref{lem:p-is-locally-morse}, choose $\C$ as in Lemma~\ref{lem:small_paths}, and choose $\M$ as in Lemma~\ref{lem:eta_i-are-morse}, then the argument above shows that $p$ is $L'$--locally $M'$--Morse. 
\end{proof}

Next we show that $p$ has relatively small intersection function.

\begin{lem}
    We have that 
    \[
    \rho_{p}(t)\leq \frac{2t}{3},
    \]
    where $\rho_{p}$ is the intersection function of $p$.
\end{lem}
\begin{proof} Let $M'$ be the Morse gauge from Lemma~\ref{lem:p-is-locally-morse} for which $p$ is locally Morse. Let $D , L'$ be the constants from Lemma~\ref{lemma:geodesicMLTG_implies_small_relators} applied to the Morse gauge $M'$ and $N = 2$. Let $L$ be large enough such that $p$ is $L'$--locally $M'$--Morse. Let $R$ be the image of a relator and $W$ a subpath (corresponding to a subword) of $R$ which is also a subpath of $p$. We can assume that $W$ is a maximal such subpath (by inclusion). Denote by $W'$ the corresponding subpath $R - W$. We have to show that $\abs{W}\leq 2\abs{R}/3$. By the choices above, we have that $W$ is $L'$--locally $M'$--Morse and if $\abs{W}> 2\abs{R}/3$, then, we have that $\abs{R}\leq D$. So it remains to show that $\abs{W}\leq2\abs{R}/{3}$ for relators $R$ with $\abs{R}\leq D$. For large enough $L$, we have that $d(y_i, x_{i+1})\geq D$ by Corollary~\ref{cor:no_overlap}. Hence, for $\abs{R}\leq D$ we know that $W$ is a subpath of 
    \[
    [y_{i-1}, x_{i}]_{\eta_{i-1}} \ast [x_i, y_i] 
    \ast [y_i,x_{i+1}]_{\eta_{i}}
    \] for some $i$. 

    Assume that $\abs{W}\geq\abs{R}/2$ and $W$ intersects $[y_{i-1}, x_{i}]_{\eta_{i-1}}$ and $[y_i,x_{i+1}]_{\eta_{i}}$ and the interior of at least one of them, say of $[y_{i-1}, x_{i}]_{\eta_{i-1}}$. Then $\abs{W'}\leq \abs{r}/2$. In particular, we can see that $x_{W'}$ and $y_{W'}$ are the endpoints of $W$. Furthermore, $d(a_i, x_{W'}) > d(a_i, x_i)$ and $d(a_i ,y_{W'}) \geq  d(a_i, y_i)$. This is a contradiction to the definition of $x_i$ and $y_i$. 
    
    It remains to show that $\abs{W}\leq 2\abs{R}/3$ if $W$ is a subword of $ [y_{i-1}, x_{i}]_{\eta_{i-1}} \ast [x_i, y_i)$ (or equivalently a subword of $(x_i, y_i] 
    \ast [y_i,x_{i+1}]_{\eta_{i}}$). Let $u$ be the word labelling $W\cap [y_{i-1}, x_{i}]_{\eta_{i-1}}$ and $v$ be the word labelling $W\cap[x_i, y_i)$. Since $\eta_{i-1}$ is a geodesic, we have that $\abs{u}\leq \abs{R}/2$. By the maximality of $W$, and since $[x_i, y_i]$ is a subword of a relator, we have that $v$ is a piece and hence $\abs{v}\leq \abs{R}/6$. Hence $\abs{W}\leq 2\abs{R}/3$, which concludes the proof. 
\end{proof}

\begin{lem}\label{lem:close-to-p}
    There exists a constant $\CC$ such that the following holds. Let $s\leq t$, then 
    \[
    d_{\mathrm{Haus}} (p[s, t], [p(s), p(t)])\leq \CC.
    \]
\end{lem}
\begin{proof}
    Let $D$ be a minimal disk diagram with boundary $\partial D = p[s, t]\ast [p(t), p(s)]$.  
    \begin{claim}\label{claim1}
        If $D$ is simple, then $\left(D, p[s, t], [p(s), p(t)]\right)$ is a combinatorial geodesic bigon.
    \end{claim}
    \textit{Proof of Claim.}
        Since $D$ is a minimal disk diagram, every interior arc of $D$ is a piece. Since $G$ is a $C'(1/9)$--group for any face $\Pi$ of $D$ and one of its interior arcs $\alpha$, we have that $\abs{\alpha}< \abs{\pi}/9$. Let $\Pi$ be a face of $D$. The above shows that if $e(\Pi) = 0$, then $i(\Pi)\geq 10$. Further, it shows that if $e(\Pi) = 1$ and its exterior arc $\alpha$ satisfies that $\abs{\alpha}\leq 2\abs{\Pi}/3$, then $i(\Pi)\geq 4$. So it remains to show that if $\alpha$ is an exterior arc contained in $p[s, t]$ or $[p(s), p(t)]$, then $\abs{\alpha}\leq 2\abs{\Pi}/3$. 
        
        Since $[p(s), p(t)]$ is a geodesic, any exterior arc $\alpha$ contained in $[p(s), p(t)]$ satisfies $\abs{\alpha}\leq \abs{\Pi}/2$. Since $\rho_p(k)\leq 2k/3$, if $\alpha$ is an exterior arc contained in $p[s, t]$, then $\abs{\alpha}\leq 2\abs{\Pi}/3$. 
    \hfill$\blacksquare$

    By Strebel's classification of bigons, this shows that (if $D$ is simple) any face $\Pi$ of $D$ intersects both $[p(s), p(t)]$ and $p[s, t]$ and has exterior arcs $\alpha\subset p[s, t]$ and $\beta\subset [p(s), p(t)]$ with $\abs{\alpha}+\abs{\beta}> 7\abs{\Pi}/9$. Since $\abs{\beta}\leq \abs{\Pi}/2$, this implies that $\abs{\alpha}\geq 2\abs{\Pi}/9$.

    With a similar argument, we can show that $p[s, t]$ is simple, since otherwise Strebel's classification of 1-gons (Lemma~\ref{lemma:strebel}) implies that there exists a face $\Pi$ with its whole boundary on $p[s, t]$ which contradicts $\rho_p(k)\leq 2k/3$. 
    
    If $D$ is not simple, then we can apply the Claim to the simple subdiagrams of $D$ to get the same result about faces of $D$.

    Now let $\Pi$ be a face of $D$ and $\alpha$ its exterior arc which is contained in $p([s, t])$. If $L$ is large enough, then $p$ is $L'$--locally $M'$--Morse for any $L'$ that we like. In particular, Lemma~\ref{lemma:geodesicMLTG_implies_small_relators} implies that there exists a constant $\CC$ such that $\abs{\Pi}\leq \CC$. Since this holds for all faces $\Pi$, we have that $d_{\mathrm{Haus}} (p([s, t]), [p(s), p(t)])\leq \CC$.
\end{proof}

\begin{cor}\label{lem:gamma-close-to-p}
    There exists a constant $\Cfive$ such that the following holds. Let $s\leq t$, then there exists a subsegment $\gamma'$ of $\gamma$ such that 
    \[
    d_{\mathrm{Haus}} (\gamma', [p(s), p(t)])\leq \Cfive.
    \]
\end{cor}

\begin{proof}
This is a direct consequence of Lemma~\ref{lem:close-to-gamma} and Lemma~\ref{lem:close-to-p}.    
\end{proof}
\begin{rem}\label{rem:gamma-is-covered}
    In fact, given any subsegment $\gamma''$ of $\gamma$, we have that $\gamma''\subset \gamma'$ for suitably chosen $s$ and $t$.
\end{rem}

\begin{cor}\label{cor:aux_path_is_qg}
There exists a constant $\Ctwo$ not depending on $\gamma$ such that the path $\gamma$ is a $\Ctwo$--quasi-geodesic.
\end{cor}
\begin{proof}
    Lemma~\ref{lem:close-to-p} and Lemma~\ref{lem:close-to-gamma} imply that increasingly large subsegments of $\gamma$ are contained in the $\Cone+\CC$--neighbourhood of a geodesic. Lemma~\ref{lemma:close-to-geodesic-implies-quasi-geodesic} concludes the proof.  
\end{proof}

\begin{cor}\label{cor:gamma-is-Morse}
    There exists a Morse gauge $M''$ not depending on $\gamma$ such that $\gamma$ is $M''$--Morse.
\end{cor}
\begin{proof}
    Let $\gamma''$ be any subsegment of $\gamma$. Use Remark~\ref{rem:gamma-is-covered} to choose $s$ and $t$ such that 
    \[
     d_{\mathrm{Haus}} (\gamma', [p(s), p(t)])\leq \Cfive,
    \]
    for a subsegment $\gamma'$ of $\gamma$ which contains $\gamma''$. Let $L'$ be a constant to be determined later. Lemma~\ref{lem:close-to-local-implies-local} implies that there exists a Morse gauge $N$ such that for large enough $L$, $[p(s), p(t)]$ is $L'$--locally $N$--Morse. If $L'$ is large enough, $G$ satisfying the geodesic MLTG implies that there exists a Morse gauge $N'$ (not depending on $\gamma$) such that $[p(s), p(t)]$ is $N'$--Morse. Lemma~\ref{new-monster}\ref{monster:hausdorff} implies that $\gamma'$ is $M''$--Morse for some Morse gauge $M''$ not depending on $\gamma, s$ or $t$. Since, $\gamma''$  is a subsegment of $\gamma'$ it is also $M''$--Morse by Lemma~\ref{new-monster}\ref{monster:subsegments}. 

    Since $\gamma''$ was chosen arbitrarily, we have shown that any subsegment of $\gamma$ is $M''$--Morse, implying that $\gamma$ itself is $M''$--Morse.
\end{proof}

\begin{proof}[Proof of Theorem~\ref{thmintro:small_cancellation}]
    Theorem~\ref{thm:MLTG_implies_sigma_cpct} shows that $(1)\implies (3)$. It is clear that $(3)\implies (2)$. Corollary~\ref{cor:aux_path_is_qg} and Corollary~\ref{cor:gamma-is-Morse} show that if $G=\pres$ is a $C'(1/9)$--group with $\sigma$-compact Morse boundary, then for every Morse gauge $M$ and constant $Q$ there exist constants $L, \Ctwo$ and a Morse gauge $M''$ such that any $L$--locally $M$--Morse $Q$--quasi-geodesic is an $M''$--Morse $\Ctwo$--quasi-geodesic, which shows that $(3)\implies (1)$.
\end{proof}

\section{Morse local-to-global implies sigma-compact}
In this section, we prove that a MLTG group has $\sigma$-compact Morse boundary. We will use tools from language theory. For this reason, let us recast some of our previous definitions in a more language theoretic fashion. Given a finite set $\mc{S}$ of formal symbols, the \emph{full language} over $\mc{S}$ is the set $\mc {S}^\ast$ consisting of all finite words over $\mc{S}$, namely all formal strings $s_1s_2\dots s_n$ such that $s_i\in \mc{S}$. Sometimes is useful to consider the  \emph{full infinite language} over $\mc{S}$, denoted by $\mc{S}^\omega$, namely the set of all infinite words on $\mc{S}$. Given an alphabet $\mc{S}$ a \emph{language} is subset $L \subseteq \mc{S}^\ast$. In what follows, we will consider a group $G$ and the alphabet $\mc{S}$ is going to be a fixed finite, symmetric set of generators for $G$.

Given a language $L \in \mc{S}^\ast$ its \emph{limit set} is the set $\Lambda (L) \subseteq \mc{S}^\omega$ defined to be the set of infinite words whose prefix are in the language, formally:
\[\Lambda(L) :=\{s_1\dots s_n \dots \in \mc{S}^\omega \mid s_1 \dots s_k \in L, \quad \forall k \in \mathbb{N}\}.\]

\begin{defn}\label{def:FSA}
    Let $\mc{S}$ be a finite set. A \emph{finite state automaton with alphabet $\mc{S}$} is a tuple $\mathcal{A}=(\Gamma,S,Z,v_0)$ where
    \begin{itemize}
        \item $\Gamma$ is a finite directed graph. We call the vertices of $\Gamma$ the \emph{states} of $\mathcal{A}$.
        \item Each edge of $\Gamma$ is labelled by an element of the alphabet $\mc{S}$.
        \item $Z$ is a subset of the vertices of $\Gamma$ and the vertices of $Z$ are called the \emph{accept states} of $\mc A$.
        \item $v_0$ is a vertex of $\Gamma$. We call $v_0$ the \emph{initial state} of $\mc A$.
    \end{itemize}
\end{defn}
Every finite state automaton produces an accepted language consisting of words produced by reading the labels of each path in the directed graph that starts at the initial state and ends at one of the accept states. We call these languages regular.
\begin{defn}\label{def:regular_language}
    Given a  finite state automaton $\mc A =(\Gamma, \mc{S},Z,v_0)$, a word $w = a_1 \dots a_n \in \mc{S}^\ast$ is read by a path $e_1, \dots, e_n$ in $\mc A$ if the label of $e_i$ is $a_i$ for each $i \in \{1, \dots, n\}$. The \emph{language accepted by $\mc A$} is the set of words in $\mc{S}^\ast$ that are read by paths in $\mc{A}$ that start at the initial state $s_0$ and end at an accepted state of $\mc{A}$. A language is \emph{regular} if it is accepted by some finite state automaton.
\end{defn}

\begin{thm}\label{thm:MLTG_implies_sigma_cpct}
    Let $G$ be a finitely generated group satisfying the MLTG property. Then the Morse boundary $\mb G$ is strongly $\sigma$-compact.
\end{thm}
\begin{proof}
Fix a Cayley graph for $G$. Since $G$ satisfies the MLTG property, for each Morse gauge $M$ there exists $B_M, Q_M$ and a Morse gauge $M'$ such that every $B_M$--locally $M$--Morse geodesic is a global $M'$--Morse $Q_M$--quasi-geodesic. Let $L_M$ be the language of all geodesic words in the fixed Cayley graph of $G$ that are $B_M$--locally $M$--Morse, and let $\Lambda (L_M)$ its limit. Note that, by the MLTG property, $\Lambda(L_M)$ contains the $M$--stratum $\ms^M G$ and is contained in the $M'$--stratum of the Morse boundary, $\ms^{M'}G$. Now consider any point $\xi$ in the Morse boundary. It is in the $M$--stratum for some $M$, and since $M$--Morse geodesic are locally $M$--Morse, $\xi \in \Lambda(L_M)$. Thus, we can write \begin{equation}\label{eq:partition_boundary_languages}
\mb G = \bigcup_{M} \Lambda(L_M) \subseteq \bigcup_M \ms^{M'}G.\end{equation}

By \cite[Theorem~3.2]{CordesRussellSprianoZalloum:regularity}, the language $L_M$ is regular, and thus there is finite state automaton $\mc A_M$ that accepts $L_M$. Let $\mc A$ be the set of all such finite state automatons. Then we can rewrite \eqref{eq:partition_boundary_languages} as \[\mb G = \bigcup_{{\mc A_M}\in \mc A} \Lambda(L_M)\subseteq \bigcup_{{\mc A_M}\in \mc A} \ms{M'}G .\]
The set $\mc A$ is countable since it is contained in the set of all oriented, finite graphs labelled by generators of $G$, and the latter is countable. So we showed that every Morse stratum $M$ is contained in one of a family of countably many Morse strata. In other words, the Morse boundary is strongly $\sigma$-compact.
\end{proof}

\section{Stationary measures}\label{sec:measure}
In this section, we will present an application of Theorem~\ref{thmintro:MLTG_implies_sigma_cpt}, which says that in most cases there is no meaningful stationary measure on the Morse boundary. 

For this section, let $G$ be a finitely generated group, $\mu$ be a probability measure whose support generates $G$ as a semigroup, and $Z$ be a metrizable topological space that $G$ acts on by homeomorphism. We first define some terminologies.  We say a probability distribution on $Z$ is \emph{$\mu$-stationary} if for any Borel set $U \subset Z$ we have
\[
    \nu(U) = \sum_{g\in G} \mu(g) \nu(g^{-1}U).
\]
The main geometric assumption we will make is the following.
\begin{defn}\label{def:has-contraction}
    A group $G$ \emph{has contraction} if there exists a geodesic metric space $X$ such that $G$ acts on $X$ geometrically and $X$ contains a strongly contracting geodesic ray.
\end{defn}
We remark that the existence of a strongly contracting geodesic ray is not known to be invariant under quasi-isometries.

\begin{thm}\label{thm:vanishing_measure}
    Suppose $G$ acts geometrically on a geodesic metric space $X$. Let $\mc{M}(X)$ be the set of Morse geodesic rays on $X$. Suppose there is a map $\rho: \mc{M}(X) \to Z$. Assume further that 
    \begin{enumerate}[label = (\alph*)]
        \item $G$ is non-hyperbolic, Morse local-to-global, and has contraction;
        \item if $\gamma_1$ and $\gamma_2$ are geodesic rays representing different elements of the Morse boundary, then $\rho[\gamma_1] \neq \rho[\gamma_2]$.
        \item $\rho$ is $G$--equivariant.
        \item $\rho\left( \mc{M}_M(X)\right)$ is compact, where $\mc {M}_M(x)$ denotes the set of $M$--Morse geodesics starting at $x_0$.
     \end{enumerate}
    Then for any $\mu$-stationary Borel measure $\nu$ on $Z$, we have
    \[
    \nu(\rho(\mc{M}(X)) )  =0.
    \]
\end{thm}

In particular, any continuous $G$--equivariant map from the Morse boundary $\mb G$ into $Z$ induces a map $\rho: \mc{M}(X) \to Z$ that satisfies the properties above, hence we have the following special case.

\begin{cor}\label{cor:mb-injects}
 Suppose $G$ is non-hyperbolic, Morse local-to-global, and has contraction; and suppose there is a continuous $G$--equivariant map $\rho \colon \mb G \to Z$. Then the image of the Morse boundary $\rho(\mb(G))$ has measure zero with respect to any $\mu$-stationary measure on $Z$.
\end{cor}

It is natural to apply Theorem~\ref{thm:vanishing_measure} to the horofunction compactification, where the existence of a $\mu$-stationary measure is guaranteed by \cite[Lemma 4.3]{mahertiozzo:random}.

\begin{defn}\label{defn:horofunction_map}
    Let $X$ be a metric space and fix a basepoint $x_0$. The \emph{horofunction} map is the map
    \begin{align*}
        \rho \colon X &\to C(X)\\
        y &\mapsto \rho_y(z) : = d(z,y) - d(x_0,y),
    \end{align*}
    where $C(X)$ denotes the set of continuous, real-valued functions on $X$. The \emph{horofunction compactification} $\ov{X}^h$ is the closure of $\rho(X)$ in $C(X)$.
\end{defn}
We want to argue that the Morse boundary of $X$ has measure zero in the horofunction compactification $\ov{X}^h$ with respect to any $\mu$-stationary measure. The issue with that statement is that in general there is no map between the Morse boundary and the horofunction compactification (the horofunction compactification does not necessarily identify fellow-travelling geodesic rays). However, we can define a map $\rho : \mc M(X) \to \ov{X}^h$ by extending the horofunction map to $\mc{M}(X)$. For a geodesic ray $\gamma$ we define $\rho[\gamma] = \lim_{t \to \infty} \rho_{\gamma(t)}$. Note that $\rho [\gamma] \in \ov{X}^h$. The group $G$ acts on $\ov{X}^h$ equivariantly (\cite[Lemma~3.4]{mahertiozzo:random}). The following lemma shows that $\rho$ satisfies the injectivity condition in Theorem~\ref{thm:vanishing_measure}.

\begin{lem}\label{lem:horofunctions_different_points}
Let $\gamma_1$ and $\gamma_2$ be geodesic rays representing different elements of the Morse boundary. Then $\rho[\gamma_1] \neq \rho[\gamma_2]$.
\end{lem}
\begin{proof}
    Since $\gamma_1$ and $\gamma_2$ represent different elements of the Morse boundary, by visibility of the Morse boundary (Lemma \ref{lem:visibility}), there is a geodesic $\eta$ such that the endpoints of $\eta$ are the endpoints of $\gamma_1, \gamma_2$ respectively (see Figure \ref{fig:injective}). In particular, there are constants $T_1, T_2$ and $K$ such that for all $t$ large enough we have 
    \[
    d(\gamma_1(t + T_1), \eta(t)) \leq K
    \qquad \text{and} \qquad d(\gamma_2(t+ T_2), \eta(-t)) \leq K.
    \]

    Choose $S$ and $s$ such that $T_1 + 2K \leq s$ and $s + 5K \leq S$, and let $z = \gamma_1(s)$, $Z = \gamma_1(S)$. Observe that for all $t \geq S$ it holds: 
    \begin{align*}
        \rho_{\gamma_1(t)}(z) - \rho_{\gamma_1(t)}(Z) 
        & = \bigl( d(z, \gamma_1(t)) - d(\gamma_1(t), x_0) \bigr) - \bigl( d(Z, \gamma_1(t))- d(\gamma_1(t), x_0) \bigr) \\ 
        &= d(\gamma_1(s), \gamma_1(t)) - d(\gamma_1(S), \gamma_1(t)) = S - s.
    \end{align*}
    Since this holds for all $t$ large enough, we obtain 
    \begin{align*}
    \rho[\gamma_1](z) -\rho[\gamma_1](Z) 
    = \lim_{t \to \infty} \rho_{\gamma_1(t)}(z) - \lim_{t \to \infty} \rho_{\gamma_1(t)}(Z) 
    = \lim_{t\to \infty} (\rho_{\gamma_1(t)}(z) - \rho_{\gamma_1(t)}(Z)) = S-s.        
    \end{align*}
    Now consider $\gamma_2$. Again, choose $t$ large enough, in particular such that $t \geq T_2 + 2K$, and let $v_t$ such $d(\eta(v_t), \gamma_2(t))\leq K$.
    We have: 
     \begin{align*}
        \rho_{\gamma_2(t)}(z) - \rho_{\gamma_2(t)}(Z) 
        & = \bigl( d(z, \gamma_2(t)) - d(\gamma_2(t), x_0) \bigr) - \bigl( d(Z, \gamma_2(t)) - d(\gamma_2(t), x_0) \bigr)  \\ 
        &= d(\gamma_1(s), \gamma_2(t)) - d(\gamma_1(S), \gamma_2(t)) \\ 
        &\leq d(\eta(-s), \gamma_2(t)) - d(\eta(-S), \gamma_2(t)) + 2K  \\ 
        &\leq  d(\eta(-s), \eta(v_t)) - d(\eta(-S), \eta(v_t)) + 4K \leq s-S + 4K.
    \end{align*}
Arguing as before, we obtain 
\[
\rho[\gamma_2](z) -\rho[\gamma_2](Z) \leq s-S + 4K.
\]
Thus, the difference $\rho[\gamma_i](z) -\rho[\gamma_i](Z) $ assumes a positive value for $i = 1$ and a negative one for $i = 2$, yielding $\rho[\gamma_1]\neq \rho[\gamma_2]$.        
\end{proof}
\begin{figure}
    \centering
    \includegraphics[width=\linewidth]{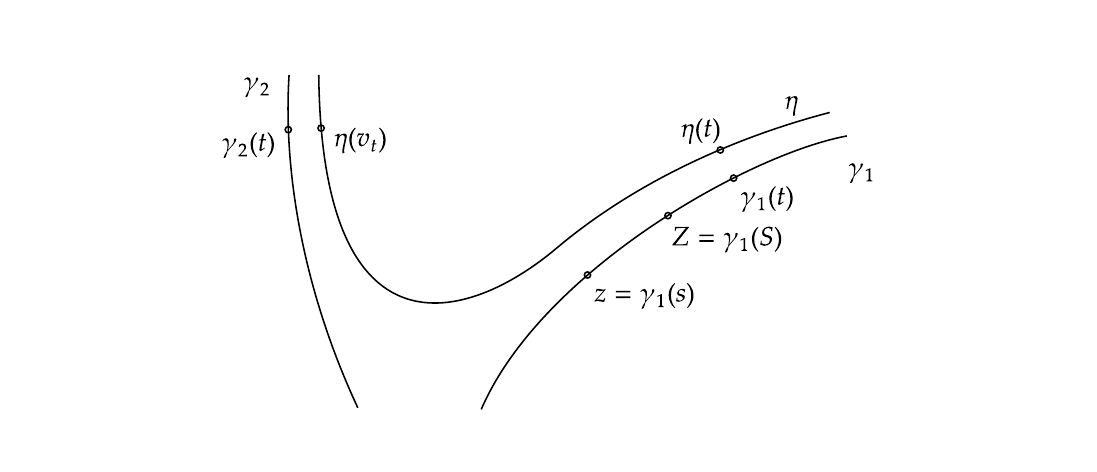}
    \caption{Geodesic rays $\gamma_1$ and $\gamma_2$ correspond to different horofunctions because $\rho_{\gamma_1(t)}(z) - \rho_{\gamma_1(t)}(Z)$ and $\rho_{\gamma_2(t)}(z) - \rho_{\gamma_2(t)}(Z)$ are different for large $t$.}
    \label{fig:injective}
\end{figure}

It remains to check that $\rho(\mc{M}(X))$ is compact. 
\begin{lem}\label{lem:horofunction-img-compact}
    Let $N$ be a Morse gauge. The image $\rho(\mc{M}_N(X))$ under $\rho$ of all $N$--Morse geodesics starting at the basepoint $x_0$ is compact in $\ov{X}^h$.
\end{lem}
\begin{proof}
     Let $\{\gamma_k\}$ be a sequence of $N$--Morse geodesic rays. By Lemma~2.1 of \cite{C:Morse}, a subsequence $\{\gamma_{k_i}\}$ converges to $\gamma$ for some $N$--Morse geodesic ray $\gamma$ in $X$. The horofuction map on $X$ is continuous by \cite[Lemma 3.2]{mahertiozzo:random}. Hence for any $t\geq 0$,
\[
	\lim_{i \to \infty} \rho_{\gamma_{k_i}(t)} (\cdot) = \rho_{\gamma(t)} (\cdot)
\]
By definition, we have the subsequence of points in the horofunction boundary $\rho_{\gamma_i}$ converges to  $\rho_{\gamma}$.
\end{proof}

Lemma~\ref{lem:horofunctions_different_points} and Lemma~\ref{lem:horofunction-img-compact} lets us apply Theorem~\ref{thm:vanishing_measure} to the horofunction map, as a result, we have 

\begin{cor}
    Suppose $G$ is non-hyperbolic, Morse local-to-global, and has contraction. Suppose $G$ acts geometrically on a geodesic metric space $X$. Then the image under the horofunction map $\rho(\mc{M}(X))$ has measure zero in the horofunction compactificaion, with respect to any $\mu$-stationary measure.
\end{cor}

We remark that the space $X$ where the horofunction compactification is defined need not to be the same as the space that contains the strongly contracting geodesic ray.

We turn our attention to the other boundaries mentioned in the introduction. The quasi-redirecting boundary satisfies the hypotheses of Theorem~\ref{thm:vanishing_measure} by \cite[Theorem~C, Theorem~E]{qingrafi:quasiredirecting}. The HHG boundary satisfies the hypotheses of Theorem~\ref{thm:vanishing_measure} by \cite[Theorem~3.4]{durhamhagensisto:boundaries}, 
\cite{Hagen:secondcountability}, and by using the proof of \cite[Theorem~4.3]{durhamhagensisto:boundaries} to show that the image of a Morse stratum is compact. 

\subsection{Proof of Theorem~\ref{thm:vanishing_measure}}
The proof of Theorem \ref{thm:vanishing_measure} relies on \cite[Lemma 4.5]{mahertiozzo:random}. We state the weaker version of the lemma which we will use. Recall that $\mu$ is a probability measure on $G$ whose support generates $G$ as a semigroup. We use the notation $\mu_n$ to denote the $n$-th convolution of $\mu$.

\begin{lem}[{\cite[Lemma~4.5]{mahertiozzo:random}}]\label{lem:-Tiozzo}
Let $Z$ be a metrizable space on which $G$ acts by homeomorphisms. Let $\nu$ be a $\mu$-stationary Borel probability measure on $Z$. Suppose that there is a measurable subset $Y \subset Z$ satisfying the following: there is a sequence of positive numbers $(\epsilon_n)_{n \in \N}$ such that for any translate $fY$ of $Y$ there is a sequence $(g_n)_{n \in \N}$ of group elements (which may depend on $f$), such that the translates $fY, g_1^{-1}fY, g_2^{-1}fY, \ldots$ are all disjoint, and for each $g_n$ there exists $m_n$ such that $\mu_{m_n}(g_n) \geq \epsilon_n$. Then $\nu(Y)=0$.
\end{lem}
We will apply the lemma and take $Y$ to be the image under $\rho$ of an $N$--stratum of the Morse boundary. We now construct the $g_i$.

\subsection*{Construction of the $g_i$}

In this section, $G$ denotes a finitely generated group acting geometrically on a geodesic metric space $X'$ satisfying the Morse-local-to-global property and containing a strongly contracting geodesic ray $\gamma:[0, \infty) \to X'$. Let $x_0'$ be a basepoint of $X'$ and $\delta>0$ be a constant such that $G\cdot x_0'$ is $\delta$--dense in $X'$. We can assume that $\gamma(0) = x_0'$. Let $C_0$ be a constant such that $\gamma$ is $C_0$--contracting. Let $C$ be a constant such that all geodesics (and their subsegments) whose endpoints are in the $\delta$--neighbourhood of a $C_0$--contracting geodesic are $C$--contracting. Such a constant $C$ exists by Lemma~\ref{lemma:hausdorff_contraction} and Lemma~\ref{lemma:subsegments}. Further, let $C'$ be the constant obtain by applying Lemma~\ref{lem:paths-are-rectangles} to the constant $C$. Lastly, fix a Morse gauge $N$. Morally, one should think of $Y$ as the $N$--stratum of $\mb X'$, but while this provides intuition for how $N$ is used in this section, we do not mention $Y$ for the rest of the section. Formally, we prove the following proposition 

\begin{prop}\label{prop:pentagon}
    Let $N$ be a Morse gauge. There exists a sequence of geodesics $\eta_i$ such that for all geodesics $\lambda$ starting at $x_0'$ and ending in $G\cdot x_0'$, there exists $i_0, i_1$ such that the following holds. Let $x, y$ be the endpoints of $\eta_i\ast_T\lambda$ and $\eta_j\ast_T\lambda$ respectively. Then if $i, j, i_0, i_1$ are all distinct, then $[x, y]$ is not $N$--Morse. Let $x', y'$ be the endpoints of $\eta_i\ast_T\lambda$ respectively $\lambda$. If $i\not \in \{i_0, i_1\}$, then $[x', y']$ is not $N$--Morse.
\end{prop}

The idea is to construct geodesics $\eta_i$ which are less and less Morse in a controlled way. We also ensure that $\eta_i$ starts at $x_0'$ and ends in $G\cdot x_0'$. We can then define $g_i$ such that $g_i\cdot x_0'$ is the endpoint of $\eta_i$.

To construct $\eta_i$, we first need to construct geodesics which are very much not Morse. We do so in the following lemma. 

\begin{lem}\label{lem:bad_geodesic}
    Let $\Cseven$ and $K_0$ be constants. There exists a constant $K\geq K_0$, a strong exhaustion $(M_n)_{n\in \N}^{\Cseven, K_0}$ and a sequence of geodesics $\hat{\beta}_n^{\Cseven, K_0}: [0, S_n]\to X'$ such that the following hold:
    \begin{enumerate}
        \item $\hat{\beta}_n^{\Cseven, K_0}$ is not $M_n$--Morse
        \item $\hat{\beta}_n^{\Cseven, K_0} [0, K]$ and $\hat{\beta}_n[S_n - K, S_n]$ are not $\Cseven$--contracting.
        \item $\hat{\beta}_n^{\Cseven, K_0}(0), \hat{\beta}_n^{\Cseven, K_0}(S_n)\in G\cdot x_0'$
    \end{enumerate}
\end{lem}
\begin{proof}
    Define $l = \delta$ and define $\Ceight$ as the constant $C_l$ from Lemma~\ref{passing-contraction}, applied to the constant $C = \Cseven$ and $l = \delta$. Let $L$ be the constant and $M$ the Morse gauge such that any geodesic which is $L$--locally $\Ceight$--contracting is $M$--Morse. Define $K = \max\{K_0, L+1\}$.
    
    By Theorem~\ref{thm:MLTG_implies_sigma_cpct} we have that $\mb X'$ is strongly $\sigma$-compact. In particular, there exists a strong exhaustion $(M_n^{\Cseven, K_0})_{n}$ of $\mb X'$. 

    By possibly passing to a subsequence of $(M_n^{\Cseven, K_0})_{n\in \N}$, we can assume by Lemma~\ref{new-monster}\ref{monster:triangle} that for any $7$-gon the following holds. If for six of the sides we have at least one of the following
    \begin{itemize}
        \item the side is $M$--Morse,
        \item the side is $M_{n}^{\Cseven, K_0}$--Morse,
        \item the side has length at most $2 + 2 \delta$,
    \end{itemize}
    then the seventh side is $M_{n+1}^{\Cseven, K_0}$--Morse. Further, we can assume that for any $4$-gon the following holds. If for three of the sides we have at least one of the following 
    \begin{itemize}
        \item the side is $M$--Morse,
        \item the side has length at most $2K' + 4$,
    \end{itemize}
    then the fourth side is $M_{1}^{\Cseven, K_0}$--Morse. 

    Since $X'$ is not hyperbolic, there exists a sequence of geodesics $\beta_n' : [0, T_n]\to X'$ such that $\beta_n'$ is not $M_{n+2}^{\Cseven, K_0}$--Morse. 

    Fix $n$. Denote by $I\subset [0, T_n - K]$ the set of all $s$ for which $\beta_n'[s, s+K]$ is not $\Ceight$--contracting. If $J\subset [0, T_n] - I$ is a connected interval, then $\beta_n'(J)$ is $L$--locally $\Ceight$--contracting and hence is $M$--Morse. By the choice of exhaustion $(M_n^{\Cseven, K_0})_n$, this implies that $\beta_n'(J)$ is $M_1^{\Ceight, K_0}$ and hence $M_{n+2}^{\Cseven, K_0}$--Morse. This shows that $I$ cannot be empty. Let $s$ and $t$ be the infimum and supremum of $I$ and let $s\leq s'\leq s+1$ and $t-1\leq t'\leq t$ be such that $\beta_n'[s', s' + K]$ and $\beta_n'[t', t'+K]$ are not $\Ceight$--contracting. Note that by the observation before, $\beta_n'[0, s-1]$ and $\beta_n'[t+1, T_n]$ are $M$--Morse. Now, if $t-s\leq 2K +2$, then the choice of exhaustion implies that $\beta_n'$ is $M_{1}^{\Cseven, K_0}$--Morse, which we know it is not. Hence $t'-s' > 2K$. We also know that $\beta_n'[s', t']$ is not $M_{n+1}^{\Cseven, K_0}$--Morse, since by the choice of exhaustion, otherwise $\beta_n'$ would be $M_{n+2}^{\Cseven, K_0}$--Morse.
    
    Let $g_nx_0'$ and $h_nx_0'$ be points in the $\delta$--neighbourhood of $\beta_n'(s')$ and $\beta_n'(t')$ respectively. Define $\hat{\beta}_n^{\Cseven, K_0}: [0, S_n]\to X'$ as a geodesic from $g_n\cdot x_0'$ to $h_n\cdot x_0'$. Since $\beta_n'[s', t']$ is not $M_{n+1}^{\Cseven, K_0}$--Morse, we have that $\hat{\beta}_n^{\Cseven, K_0}$ cannot be $M_{n}^{\Cseven, K_0}$--Morse. It remains to show that $\hat{\beta}_n^{\Cseven, K_0}[0, K]$ and $\hat{\beta}_n^{\Cseven, K_0}[S_n - K, S_n]$ are not $\Cseven$--contracting. This follows directly from the choice of $\Ceight$, Lemma~\ref{passing-contraction} and the fact that $\beta_n'[s', s' +K]$ and $\beta_n'[t' - K, t']$ are not $\Ceight$--contracting.
\end{proof}

Next we need to construct geodesics which are $C$--contracting and have endpoints in $G\cdot x_0'$. This is much easier. Namely, for each $A\geq 0$ there exists a point $x\in G\cdot x_0'$ with $d(x, \gamma(A+\delta))\leq \delta$. We can define $\hat{\gamma}_A$ as a geodesic from $x_0'$ to $x$. We have that 
\begin{enumerate}
    \item $\hat{\gamma}_A$ is $C$--contracting. Recall that $\gamma$ is $C_0$--contracting, so this follows directly from the definition of $C$.
    \item The domain of $\hat{\gamma}_A$ is $[0, A + \delta_A]$ for some $0\leq \delta_A \leq 2\delta$.
\end{enumerate}

    We use the notation $\ast_T$ to denote the \emph{translated concatenation}, that is, for two paths $\alpha([s_1,s_2])$ and $\beta([t_1,t_2])$, with endpoints in $G\cdot x_0'$ we write 
    \[
        \alpha([s_1,s_2]) \ast_T \beta([t_1,t_2]) = \alpha([s_1,s_2]) \ast \big(  \alpha(s_2) \beta(t_1)^{-1} \cdot\beta([t_1,t_2])\big),
    \]
 where, with an abuse of notation, we denote by $\alpha(s_2)$, resp. $\beta(t_1)$, an element of $G$ sending $x_0'$ to $\alpha(s_2)$, resp. $\beta(t_1)$.
\begin{lem}\label{lemma:pab-quasi-geodesic}
    There exist constants $\Cd$, $\Cdd$, and $Q$ such that the following holds. For all $\Cddd$, for $A$ and $B$ large enough, the following hold.
    Consider the path
    \[
        p_{A, B} = \hat{\gamma}_A\ast_T\hat{\beta}_B^{\Cdd, 1}\ast_T \hat{\gamma}_A
    \]
   For simplicity, we write $p_{A,B} = \gamma^1 \ast \beta^1 \ast \gamma^2$, where $\gamma^1$ and $\gamma^2$ denote the first and last $\hat{\gamma}_A$ subsegments, and $\beta^1$ denotes the $\hat{\beta}_B^{\Cdd, 1}$ subsegment. We have that
    \begin{enumerate}
        \item $\diam\{\pi_{\gamma^1}(\beta^1)\}\leq \Cd$,\label{small-proj-1}
        \item $d(\gamma^1, \gamma^2) \geq \Cddd$,  \label{largedist}
        \item $\diam\{\pi_{\gamma^1}(\gamma^2)\}\leq \Cd$,
        \item $p_{A,B}$ is a $Q$--quasi-geodesic.\label{Q}
    \end{enumerate}
\end{lem}

\begin{proof}
    Recall that $C'$ is the constant obtained from applying Lemma~\ref{lem:paths-are-rectangles} to $C$ and that $\hat{\gamma}_A$ is $C$--contracting. Let $C_1$ be the constant such that any geodesic (and its subsegments) with endpoints in the $C'$--neighbourhood of a $C$--contracting geodesic is $\Cdd$--contracting. Such a constant $\Cdd$ exists by Lemma~\ref{lemma:hausdorff_contraction} and Lemma~\ref{lemma:subsegments}. Let $K$ be the constant obtained when applying Lemma~\ref{lem:bad_geodesic} to the pair $\Cseven = \Cdd$ and $K_0 = 1$.  Let $x\in \gamma^1$ and $y \in \beta^1$. Let $a = \gamma^1(A+ \delta_A)$ and $b = \gamma^2 (0)$ be the transition points between segments. With this notation, we know that $\beta^1[0, K]$ and $\beta^1[S_B - K, S_B]$ are not $\Cdd$--contracting.

    (1): Let $q\in \beta^1$ be the last point with $d(q, \gamma^1)\leq C'$. By Lemma~\ref{lemma:hausdorff_contraction} we know that $[a, q]_{\beta^1}$ is $\Phi_d(C') = C_1$--contracting. If $d(a, q)\geq K$, we know by Lemma~\ref{lemma:subsegments}, that $\beta^1[0, K]$ is $\Cdd$--contracting, a contradiction. Hence $d(a, q)< K$. Lemma~\ref{lem:paths-are-rectangles} implies that 
    \begin{align*}
        \diam\{\pi_{\gamma^1}(\beta^1\}\leq K+ C'.
    \end{align*}
    So \eqref{small-proj-1} holds for $\Cd = K + C'$.
    
    (2): By Lemma~\ref{lem:theorem1.4}, there exists a Morse gauge $M$ such that any $C$--contracting geodesic is $M$--Morse. For a fixed constant $\Cddd$, there exists a Morse gauge $M'\geq M$ such that any geodesic segment of length at most $\Cddd$ is $M'$--Morse. Assume that $d(\gamma^1, \gamma^2)\leq \Cddd$. Choose $x\in \gamma^1$ and $y\in \gamma^2$ which realize the distance between $\gamma^1$ and $\gamma^2$. Now there exists a 4-gon with sides $\beta^1, [b, y]_{\gamma^2}, [x, y]$ and $[y, a]_{\gamma^1}$. Since three of the sides are $M'$--Morse, the third side is $M''$--Morse for a Morse gauge $M''$ only depending on $M'$. Since $(M_n^{\Cdd, 1})_{n\in \N}$ is a strong exhaustion we know that there exits $n$ such that any $M''$--Morse geodesic is $M_n^{\Cdd, 1}$--Morse. So for $B\geq n$, this is a contradiction to $\beta^1$ not being $M_n^{\Cdd, 1}$--Morse.

    (3): This follows from (2) applied to $\Cddd = C'$ and Lemma~\ref{lem:paths-are-rectangles}.

    (4): Let us summarize what we know so far, namely, for large enough $B$ we have that $\pi_{\gamma^1}(\beta^1)\subset \mc N_{\Cd} (a)$ and by symmetry, $\pi_{\gamma^2}(\beta^1)\subset \mc N_{\Cd} (b)$. Hence (3) shows that $\pi_{\gamma^1}(\gamma^2)\subset \mc N_{2\Cd} (a)$ and by symmetry $\pi_{\gamma^2}(\gamma^1)\subset \mc N_{2\Cd} (b)$. Hence, any path from $x\in \gamma^1$ to $y\in \beta^1$ passes through the $C'+\Cd$--neighbourhood of $a$. Namely, either $d(x, a)\leq C' + \Cd$ and the statement is immediate $d(x, a)\geq C' + \Cd$ and then Lemma~\ref{lem:paths-are-rectangles} implies that $[x, y]$ passes through the $C'$--neighbourhood of $\pi_{\gamma^1}(y)$, which is at distance at most $\Cd$ from $a$. Analogously, any path from $x\in \gamma^1$ to $z\in \gamma^2$ passes through the $2\Cd + C'$--neighbourhood of $a$ and $b$. The statement follows from the triangle inequality. 
\end{proof}

Let $\Cdd, \Kd$ be the constants from the lemma above. From now on, we denote $\hat{\beta}_B^{\Cdd, \Kd}$ by $\hat{\beta}_B$ and $M_n^{\Cdd, \Kd}$ by $M_n$. By potentially passing to a subsequence of $(M_n)_{n\in \N}$ we can assume that if we have a $6$-gon where each of the sides except one is one of the following 
\begin{itemize}
    \item $C$--contracting,
    \item $N$--Morse, where $N$ fixed at the beginning of the subsection,
    \item has length at most $2C'$, where $C'$ denotes the constant of Lemma~\ref{lem:paths-are-rectangles} applied to $C$,
\end{itemize}
then the last side is $M_1$--Morse.

    \begin{lem}\label{lem:local-morsepab}
        For all $B$, there exists a Morse gauge $M_B$ such that for all $A$, $p_{A, B}$ is $M_B$--Morse.
    \end{lem}
    \begin{proof}
        Since every finite segment is Morse for some Morse gauge, we have that $\hat{\beta}_B$ is $N_B$--Morse for some Morse gauge $N_B$. Lemma~\ref{new-monster}\ref{monster:concatenation} about concatenation of Morse quasi-geodesics concludes the proof.  
    \end{proof}

    Let $L_B, M_B'$ and $Q_B$ be constants such that all $L_B$--locally $M_B$--Morse $Q$--quasi-geodesics are $M_B'$--Morse $Q_B$--quasi-geodesics.

    \begin{lem}\label{lem:q-is-quasi-geodesic-for-large-a}
        For all $B$ large enough there exists a constant $L_B$ such that for all $A\geq L_B$, the following holds. The path
        \begin{align*}
        q_{A,B} = \underbrace{\hat{\gamma}_A\ast_T \hat{\beta}_B\ast \hat{\gamma}_A\ast_T\ldots \ast \hat{\beta}_B\ast_T\hat{\gamma}_A}_{\text{$2\cdot \anz - 1$ times}}
        \end{align*}
        is an $M_B'$--Morse $Q_B$--quasi-geodesic.
    \end{lem}

    We sometimes write $q_{A, B} = \gamma^1\ast \beta^1 \ast  \ldots \beta^{\anz -1}\ast \gamma^{\anz}$, where the $\gamma^i$ and $\beta^i$ are the translates of $\hat{\gamma}_A$ and $\hat{\beta}_B$ respectively. 
    
    \begin{proof}
    Lemma~\ref{lemma:pab-quasi-geodesic}~\eqref{Q} yields that for all $A$ and $B$ large enough, the path $p_{A, B} = \hat{\gamma}_A \ast_T \hat{\beta}_B \ast_T \hat{\gamma}_A$ is a $Q$--quasi-geodesic. Hence by Lemma~\ref{lem:local-morsepab}, $q_{A,B}$ is $A$--locally an $M_B$--Morse quasi-geodesic. In particular, if $A\geq L_B$, the path $q_{A, B}$ is $L_B$--locally an $M_B$--Morse $Q$--quasi-geodesic. Hence by the Morse-local-to-global property, we have that $q_{A, B}$ is an $M_B'$--Morse $Q_B$--quasi-geodesic.
    \end{proof}

    \begin{lem}\label{lem:small-projection-qab-to-itself} 
       There exists a constant $R$, such that for large enough $B$ and $A$ large enough compared to $B$ the following holds for all $1\leq i \leq \anz$. 
        \begin{align*}
            \diam\{\pi_{\gamma^i} (\gamma^1\ast \beta^1\ast\ldots \ast \beta^{i-1})\}\leq R.
        \end{align*}
    \end{lem}
    \begin{proof}
        Let $\Cd$ be the constant from Lemma~\ref{lemma:pab-quasi-geodesic}. Lemma~\ref{lemma:pab-quasi-geodesic} implies that
        \begin{align*}
             \diam\{\pi_{\gamma^i} (\gamma^{i-1}\ast \beta^{i-1})\}\leq 2\Cd.
        \end{align*}
        Let $D$ be the constant such that any $C$--contracting geodesic has $D$--bounded geodesic image, such a constant exists by Lemma~\ref{lemma:equivalence_strong_contraction1}. Since $q_{A, B}$ is a $Q_B$--quasi-geodesic,  for large enough $A$ we have $d(\gamma^j, \gamma^i) > D$ and $d(\beta^j, \gamma^i) > D$ for all $j< i-1$. Hence 
        \begin{align}
            \diam\{\pi_{\gamma^i} (\gamma^{j})\}\leq D \quad \text{and}\quad  \diam\{\pi_{\gamma^i} (\beta^{j})\}\leq D
        \end{align}
        for all $j<  i-1$. Choosing $R = 2\Cd + 2\cdot \anz D$ concludes the proof.
    \end{proof}
 
        Denote by $\eta_{A, B}$ a geodesic from the starting point of $q_{A,B}$ to the endpoint of $q_{A, B}$. The lemma above implies that the distance between $q_{A, B}$ and $\eta_{A, B}$ can be bounded in terms of $M_B'$. The lemma below strengthens this and shows that if any geodesic $\lambda$ has large projection onto $\eta_{A, B}$, then it enters the $C'$--neighbourhood of some $\gamma^i$. 
        
    \begin{lem}\label{lemma:large-proj-implies-fellow}
        There exists a constant $\T$ such that the following holds. For all $B$ large enough, there exists a constant $D_B$ such that for all $A$ large enough compared to $B$ the following holds. 
        If a geodesic $\lambda$ satisfies
        \begin{align*}
            \diam\{\pi_{\eta_{A, B}}(\lambda)\}\geq 4A +D_B,
        \end{align*}
        then there exists $1\leq i < \anz$ such that $\gamma^i[\T, A - \T]$ and $\gamma^{i+1}[\T, A- \T]$ are at $C'$--bounded Hausdorff distance from a subgeodesic of $\lambda$, where $C'$ is the constant of Lemma~\ref{lem:paths-are-rectangles} applied to $C$. Furthermore, if $\lambda$ and $\eta_{A, B}$ have the same starting (respectively end) point, then the statement holds for $i=1$ (respectively $i = \anz -1$) and there is a prefix (respectively) suffix of $\lambda$ which is at Hausdorff distance at most $C'$ from $\gamma^1[0, A - T_0]$ (respectively $\gamma^{\anz}[T_0, A + \delta_A]$).
    \end{lem}
    \begin{proof}
        Assume that $\diam\{\pi_{\eta_{A, B}}(\lambda)\}\geq 4A +D_B$ for some constant $D_B$ to be determined later.
        Let $x,y$ be points in $\lambda$ and $x'\in \pi_{\eta_{A, B}}(x)$ and $y'\in \pi_{\eta_{A, B}(y)}$ such that $d(x', y')\geq 4A + D_B$. Recall that for $B$ large enough and $A$ large enough compared to $B$, $q_{A, B}$ is a $M_B'$--Morse $Q_B$--quasi-geodesic. Let $D_1 = M_B'(1)$ and let $D_2$ be the bound on the Hausdorff distance when applying Lemma~\ref{new-monster}\ref{monster:hausdorff} to $M = M_B'$ and $Q = D_1$, let $D_3 = 2(D_2 + D) +1$. Let $x'', y''$ be the points on $[x', y']_{\eta_{A, B}}$ with $d(x', x'') = d(y', y'') =  D_3$. With this setup, there exist $s\leq t$ such that 
        \[
        d(q_{A, B}(s), x'')\leq D_1, d(q_{A, B}(t), y'')\leq D_1 \qquad \text{and} \qquad 
        d_{\mathrm{Haus}}(q_{A, B}[s, t], [x'', y'']_{\eta_{A, B}})\leq D_2.
        \]
        In particular, 
        \[
        d(q_{A, B}[s, t], [x_0', x']_{\eta_{A, B}}\cup [y', b]_{\eta_{A, B}})\geq D_3,
        \] 
        where $b$ denotes the endpoint of $\eta_{A, B}$. Recall that $x_0'$ is the starting point of $\eta_{A, B}$.
        Further, we have that 
        \[
        d(x'', y'')\geq 4A + D_B - 2(D_2 + D_3),
        \]
        and hence 
        \[
        t - s\geq 4A + D_B - 2(D_1 +D_2 + D_3).
        \]
        Choose $D_B$ such that 
        \[
        D_B - 2(D_1 +D_2 + D_3)  =  2\abs{\hat{\beta}_B}.
        \]
        Consequently, there exists $i$ such that $\gamma^i\ast \beta^i\ast\gamma^{i+1}\subset q_{A, B}[s, t]$. Observe that the path
        \[
        \zeta = \beta^{i}\ast\gamma^{i+1}\ast \ldots\ast \gamma^{\anz} \ast [b, y']_{\eta_{A, B}} \ast[y', y]\ast [y, x]_{\lambda^{-1}}\ast [x, x']\ast [x', x_0']_{\eta_{A, B}}\ast \gamma_1\ast \beta^1\ast \ldots\ast\gamma^{i-1}\ast \beta^{i-1}   
        \]
        as depicted in Figure~\ref{fig:large-projection} is a path from one endpoint of $\gamma^i$ to its other endpoint. In particular
        \[
        \diam\{ \pi_{\gamma^i}(\zeta)\}\geq A.
        \]
        We now proceed to bound the projection of subpaths of $\zeta$ to show that $\lambda$ has a large projection onto $\gamma^i$ and hence has to come close to $\gamma^i$. Lemma~\ref{lem:small-projection-qab-to-itself} shows that 
        \begin{align*}
            \diam\{ \pi_{\gamma^i}(\beta^i\ast\gamma^{i+1}\ast \ldots\ast \gamma^{\anz} )\}\leq R \quad \text{and} \quad  \diam\{ \pi_{\gamma^i}(\gamma^1\ast \beta^1\ast \ldots\ast\gamma^{i-1}\ast \beta^{i-1})\}\leq R.
        \end{align*}
        We have shown that $d(\gamma^i, [x_0', x']_{\eta_{A, B}}\cup [y', b]_{\eta_{A, B}})\geq D_3\geq D$ and since $\gamma^i$ satisfies the $D$--bounded geodesic image property, we have that 
         \begin{align*}
            \diam\{ \pi_{\gamma^i}([a, x']_{\eta_{A, B}} )\}\leq D \quad \text{and} \quad  \diam\{ \pi_{\gamma^i}([y', b]_{\eta_{A, B}})\}\leq D
        \end{align*}

        \begin{figure}
            \centering
            \includegraphics[width=1\linewidth]{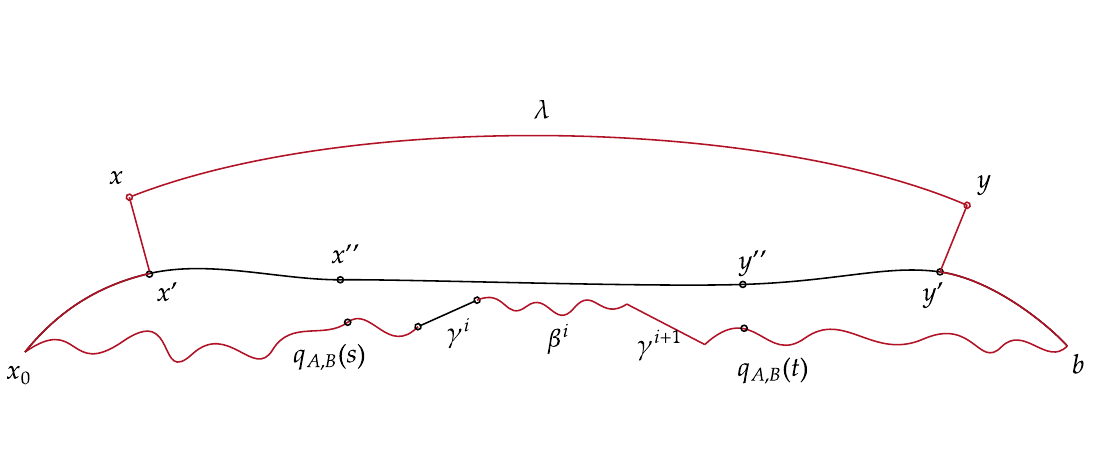}
            \caption{Depiction of the path $\zeta$.}
            \label{fig:large-projection}
        \end{figure}
        
        Next we show that $d(\gamma^i, [x, x'])\geq D$. Assume that it is not the case, then there exists $z\in [x, x']$ with $d(z, \gamma^i)\leq D$. Since $\gamma^i$ is contained in the $D_2$--neighbourhood of $[x'', y'']_{\eta_{A, B}}$, we have that 
        \[ 
        d(z, z')\leq D_2 + D \quad \text{for some $z'\in [x'', y'']_{\eta_{A, B}}$.}
        \]
        Recall that 
        \[d(x'', x') = D_3 = 2D_2 + 2D +1,\]
        so by the triangle inequality, we have that $d(z, z') < d(z, x')$, which is a contradiction to $x'\in \pi_{\eta_{A, B}}(x)$. Thus indeed, $d(\gamma^i, [x, x'])\geq D$, which implies that 
        \begin{align*}
            \diam\{ \pi_{\gamma^i}([x, x'])\}\leq D.
        \end{align*}
        Analogously, we can show that 
        \begin{align*}
            \diam\{ \pi_{\gamma^i}([y, y'])\}\leq D.
        \end{align*}

        The bounds above imply that
        \begin{align*}
            \diam\{ \pi_{\gamma^i}([x, y]_\lambda)\}\geq A - 3D - 2R.
        \end{align*}
        Lemma~\ref{lem:paths-are-rectangles} shows that the main part of the statement holds for $T_0 = 3D + 2R$. 

        If $\lambda$ and $\eta_{A, B}$ both start (respectively end) at the same point $x_0'$ (respectively $b$), then we can instead define $x = x' = x_0'$ (respectively $y = y' = b$) and a similar proof will show the desired statement.       
    \end{proof}

    \begin{lem}\label{lem:length-of-eta}
        Let $B$ be large enough. For $A$ large enough compared to $B$, we have that $\abs{\eta}\geq (\anz - 1)A$.
    \end{lem}
    
    \begin{proof}
        For $B$ large enough and $A$ large enough compared to $B$, Lemma~\ref{lem:small-projection-qab-to-itself} shows that for each $1\leq i \leq \anz$, the diameter of the projection of $\eta_{A, B}$ onto $\gamma_i$ is at least $A - 2R$ and hence by Lemma~\ref{lem:paths-are-rectangles} $C'$--fellow travels $\gamma_i$ for at least $A - 2R - 2C'$. For large enough $B$, Lemma~\ref{lemma:pab-quasi-geodesic}~\eqref{largedist} shows that the segments where $\eta_{A, B}$ fellow travels $\gamma^i$ and $\gamma^j$ are disjoint for $i\neq j$. Hence in total, 
        \[
        \abs{\eta_{A, B}}\geq \anz(A - 2R - 2C')\geq (\anz - 1)A
        \]
        for large enough $A$.
    \end{proof}
    The following is an immediate consequence of the previous two lemmas.
    \begin{cor}\label{cor:eta-stays-close}
    For large enough $B$ and for $A$ large enough compared to $B$, we have that $\gamma^i[\T, A - \T]$ is contained in the $C'$--neighbourhood of $\eta_{A ,B}$ for all $i$. In particular, if $x\in \eta_{A, B}$ with $d(x, x_0')\leq A/2$, then $d(x, \gamma_1)\leq C'$.
    \end{cor}

    \begin{lem}\label{lem:morse-projects-short}
        Let $\lambda$ be an $N$--Morse geodesic. For large enough $B$, and for $A$ large enough compared to $B$ we have that $\diam(\pi_{\eta_{A, B}}(\lambda))\leq 4A + D_B$.
    \end{lem}
    Recall that $N$ is the Morse gauge that we fixed earlier.
    \begin{proof}
        Assume that $B$ is large enough and $A$ large enough compared to $B$ such that Lemma~\ref{lemma:large-proj-implies-fellow} holds. Assume by contradiction that $\diam(\pi_{\eta_{A, B}}(\lambda)) > 4A + D_B$.  By Lemma~\ref{lemma:large-proj-implies-fellow} there exists $i$ such that $d(\lambda ,\gamma^i)\leq C'$ and $d(\lambda, \gamma^{i+1})\leq C'$. Let $x, y\in \lambda$ and $x'\in \gamma^i$, $y'\in \gamma^{i+1}$ be points with $d(x, x')\leq C'$ and $d(y, y')\leq C'$. Let $a$ and $b$ be the endpoints of $\beta^i$. Consider the $6$--gon with sides $\beta^i, [b, y]_{\gamma^{i+1}} ,[y', y], [y, x]_{\lambda}, [x, x']$ and $[x', a]_{\gamma^i}$ as depicted in Figure~\ref{fig:hexagon}. The 6-gon has two edges which have length at most $C'$ (and hence Morse for a Morse gauge only depending on $C'$), two sides which are $C$--contracting, one side is $N$--Morse and the remaining side is $\beta^i$. The choice of $(M_n)_{n\in \N}$ implies that $\beta^i$ and hence its translate $\hat{\beta}_B$ is $M_1$--Morse, which is a contradiction to $\hat{\beta}_B$ not being $M_B$--Morse by construction.
    \end{proof}  

\begin{figure}
    \centering
    \includegraphics[width=\linewidth]{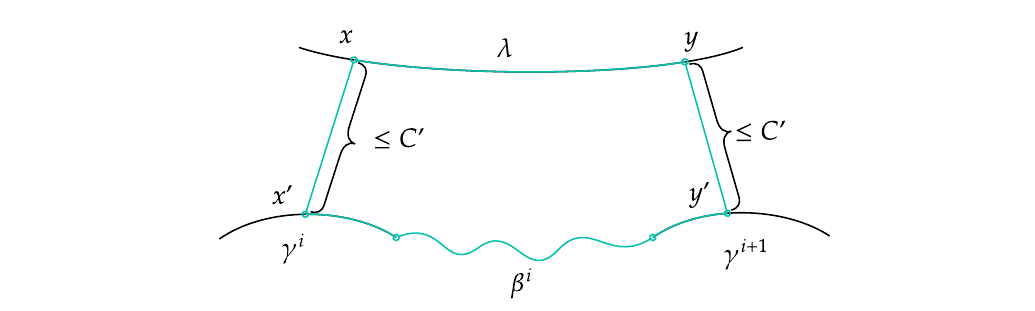}
    \caption{In the hexagon, $[x,x']$ and $[y,y']$ have length at most $C'$, the sides contained in $\gamma^i$ and $\gamma^{i+1}$ are $C$--contracting, and $[x,y]_\lambda$ is $N$--Morse. Hence the remaining side $\beta^i$ is $M_1$--Morse.}
    \label{fig:hexagon}
\end{figure}

We now construct a sequences $(B_i)_{i\in \N}$ and $(A_i)_{i\in \N}$ as follows. Let $(B_i)_{i\in \N}$ be a sequence such that 
    \begin{itemize}
        \item $B_1$ is large enough to satisfy Lemma~\ref{lem:q-is-quasi-geodesic-for-large-a}, \ref{lem:small-projection-qab-to-itself}, \ref{lemma:large-proj-implies-fellow}, \ref{cor:eta-stays-close}, \ref{lem:morse-projects-short} and Corollary~\ref{lem:length-of-eta}
        \item $B_{i+1}=B_i+1$ (and hence the $B_i$ are also large enough to satisfy the lemmas mentioned above).
    \end{itemize}
Let $(A_i)_{i\in \N}$ be a sequence such that the following hold. 
    \begin{itemize}
        \item All $A_i$ are large enough compared to $B_i$ so that they satisfy the lemmas mentioned above.
        \item We have $A_i\geq \anz(A_{i-1} + \abs{\hat{\beta}_{B_{i-1}}} + 2\delta + C'+T_0+ D_{B_i})+1 + L_{B_i}$.
    \end{itemize}
    For all $i$, define 
    \begin{align*}
        \eta_i = \eta_{A_i, B_i}.
    \end{align*}
    From now on, we write
    \begin{align*}
        q_{A_i, B_i} = \gamma_i^1\ast \beta_i^1\ast \ldots\ast \gamma_i^{\anz}.
    \end{align*}
    Where $\gamma_i^j$ and $\beta_i^j$ are the appropriate translates of $\hat{\gamma}_{A_i}$ and $\hat{\beta}_{B_i}$.
    
    \begin{lem}\label{lem:small-pairwise-projection-of-etas}
        For any $i \neq j$, we have that
        \[
        \diam(\pi_{\eta_i}(\eta_j)) \leq 4A_i + D_{B_i}
        \]
    \end{lem}
     \begin{proof}
        If $i > j$, we have by construction that $A_i\geq\anz (A_{j} + 2\delta + \abs{\hat{\beta}_{B_{j}}}) \geq\abs{\eta_j}$ and hence 
        \[
        \diam(\pi_{\eta_i}(\eta_j)) \leq 2\abs{\eta_j}\leq 2A_i.
        \]
        It remains to show the statement for $i < j$. Assume by contradiction that $\diam(\pi_{\eta_i}(\eta_j)) > 4A_i + D_{B_i}$. By Lemma~\ref{lemma:large-proj-implies-fellow}, there exists $x\in \eta_j$ and $y\in \gamma_{i}^2$ such that $d(x, y)\leq C'$. By the triangle inequality, we have that 
        \[
        d(x, x_0')\leq 2A_i + 2\delta_{A_i} + \abs{\hat{\beta}_{B_{i}}} + C'\leq A_j/2.
        \]
        Hence by Corollary~\ref{cor:eta-stays-close} we have that $d(x, z)\leq C'$ for some $z\in \gamma_j^1$. Let $u$ be the starting point of $\gamma_i^2$. Consider the $5$-gon with sides $\gamma_i^1, \beta_i^1, [u, y]_{\gamma_i^2}, [y, z]$ and $[z, x_0']_{\gamma_{j}^1}$, as depicted in Figure~\ref{fig:projection-etas}. Since each of the sides except $\beta_{i}^1$ is either $C$--contracting or has length at most $2C'$, the choice of $(M_n)_{n\in \N}$ implies that $\beta_i^1$ is $M_1$--Morse, which is a contradiction to it not being $M_{B_i}$--Morse by construction.
\begin{figure}
        \centering
        \includegraphics[width=1\linewidth]{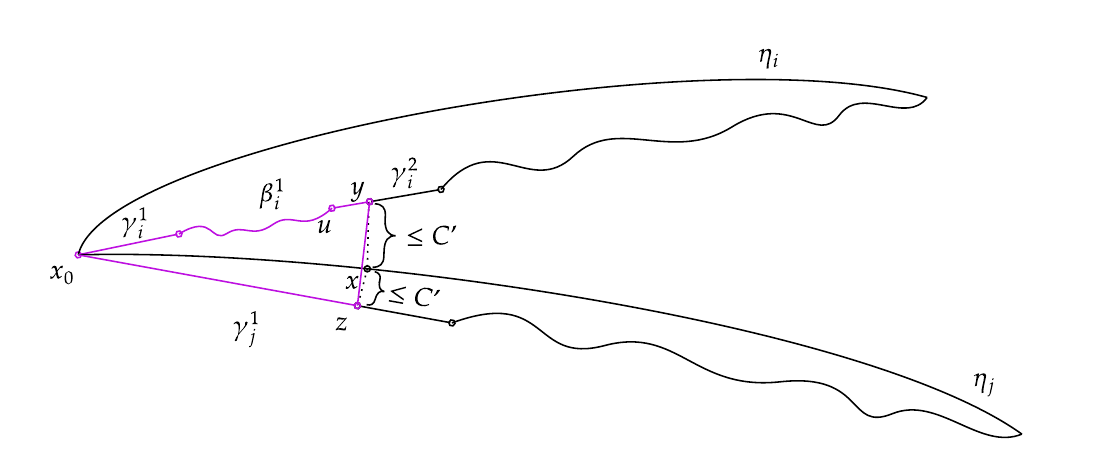}
        \caption{If the diameter of the projection $\pi_{\eta_i}(\eta_j)$ is too large, then the pentagon in the figure would force $\beta_i^1$ to be Morse, contradicting the construction.}
        \label{fig:projection-etas}
    \end{figure}
        \end{proof}

    \begin{lem}\label{lem:at-most-one-large-projection}
        Let $\lambda$ be any geodesic starting at $x_0'$. Then there exists at most one index $i$ such that
        \begin{align*}
            \diam{\pi_{\eta_i}(\lambda_i)}\geq 4A_i + D_{B_i},
        \end{align*}
        where $\lambda_i$ is the translate of $\lambda$ starting at the endpoint of $\eta_i$. Moreover, there exists at most one index $i$ such that 
        \begin{align*}
            \diam{\pi_{\eta_i}(\lambda)}\geq 4A_i + D_{B_i}.
        \end{align*}
    \end{lem}
    The proof of Lemma~\ref{lem:at-most-one-large-projection} is analogous to the proof of Lemma~\ref{lem:small-pairwise-projection-of-etas}.
   \begin{proof}
        We prove the first part of the statement, and the moreover part follows from an analogous proof. Let $\tilde{\eta}_j$ be a translate of $\eta_j$ such that $\tilde{\eta}_j$ and $\eta_i$ have the same endpoint, which we denote by $b$. We denote by $\tilde{\gamma}_j^k$ and $\tilde{\beta}_j^k$ the translates of $\gamma_j^k$ and $\beta_j^k$ such that $\tilde{\eta}_j = \tilde{\gamma}_j^1\ast\tilde{\beta}_j^1\ast \ldots\ast \tilde{\gamma}_j^{\anz}$. We have to prove that at most one of the following holds,
       \begin{align*}
            \diam{\pi_{\eta_i}(\lambda_i)}\geq 4A_i + D_{B_i}\quad \text{or}\quad   \diam{\pi_{\tilde{\eta}_j}(\lambda_i)}\geq 4A_j + D_{B_j}.
        \end{align*}
        Without loss of generality, we may assume that $i > j$. Lemma~\ref{lemma:large-proj-implies-fellow} implies the following 
        \begin{itemize}
            \item there exists a suffix $\lambda'$ of $\lambda_i$ which is at Hausdorff distance at most $C'$ from $\gamma_i^{\anz}[T_0, A_i + \delta_{A_i}]$. 
            \item there exists a point $x\in \lambda_i$ and $y = \tilde{\gamma}_j^{\anz - 1}(t)$ such that $d(x, y)\leq C'$. 
        \end{itemize}
        By the triangle inequality, we have that $d(x, b)\leq C' + 2A_i + B_i$. Since $\lambda'$ is at Hausdorff distance at most $C'$ from $\gamma_i^{\anz}[\T, A_i]$, it has length at least $A_i - T_0 - 2C' >  C' + 2A_j + B_j$. Consequently, $x\in \lambda'$ and there exists $z\in \gamma_i^{}$ with $d(x, z)\leq C'$. Now consider the $5$--gon with sides $\tilde{\gamma}_j^{\anz - 1}[t, A_j + \delta_{A_j}], \tilde{\beta}_j^{\anz -1}, \tilde{\gamma}_j^{\anz}, [b, z]_{\gamma_i^{\anz}}$ and $[z, y]$. Each of these sides except $\tilde{\beta}_j^{\anz -1}$ is either $C$--contracting or has length at most $2C'$. Hence by the choice of $M_1$, we know that $\tilde{\beta}_j^{\anz -1}$ and hence its translate $\hat{\beta}_{B_j}$ is $M_1$--Morse and hence $M_{B_j}$--Morse. A contradiction to $\beta_{B_j}$ not being $M_{B_j}$--Morse by construction.
    \end{proof}

    Fix a geodesic $\lambda$. If they exists, let $i_0, i_1$ be the unique integers such that
    \begin{align*}
        \diam{\pi_{\eta_{i_0}}(\lambda_{i_0})}\geq 4A_{i_0} + D_{B_{i_0}},\\
        \diam{\pi_{\eta_{i_1}}(\lambda)}\geq 4A_{i_1} + D_{B_{i_1}},
    \end{align*}

    where $\lambda_i$ again denotes the appropriate translate of $\lambda$ starting at the endpoint of $\eta_i$.
    
    \begin{lem}\label{lem:strange-projection}
        For all $i\neq j\in \N\setminus\{i_0, i_1\}$, we have that $\pi_{\eta_i}(\lambda_j)\leq 4A_i + D_{B_i}$.
    \end{lem}
    \begin{proof}
        If $i = j$, this is implied by Lemma~\ref{lem:at-most-one-large-projection}. So from now on we assume that $i\neq j$.
        Assume that 
        \[
        \pi_{\eta_i}(\lambda_j)> 4A_i + D_{B_i}.
        \]
        Lemma~\ref{lemma:large-proj-implies-fellow}, implies that there exists $k$ such that $\gamma_i^k[ \T, A_i -  \T]$ is in the $C'$--neighbourhood of $\lambda_j$. Corollary~\ref{cor:eta-stays-close} implies that $\gamma_i^k[\T, A_i - \T]$ is also in the $C'$--neighbourhood of $\eta_i$. In particular, there exists points $x\in \eta_i$ and $y\in \lambda_j$ such that $d(x, y)\leq 2C'$. Let $b$ be the endpoint of $\eta_j$. Our observation above shows that the path $p = [x_0', x]_{\eta_i}\ast [x, y]\ast[y, b]_{\lambda_j}$ connects the two endpoints of $\eta_j$. Lemma~\ref{lem:small-pairwise-projection-of-etas} implies that 
        \[
             \diam(\pi_{\eta_j}([x_0', x]_{\eta_i})) \leq 4A_j + D_{B_j}.
        \]
        Lemma~\ref{lem:at-most-one-large-projection} implies that 
        \[
             \diam{\pi_{\eta_j}([y, b]_{\lambda_j})}\leq 4A_j + D_{B_j}.
        \]
        Consequently, 
        \[
         \diam{\pi_{\eta_j}([x, y])}\geq \abs{\eta_j} - 2( 4A_j + D_{B_j}) > 4A_j + D_{B_j}.
        \]
        By Lemma~\ref{lemma:large-proj-implies-fellow}, there exists an integer $k$ and a subsegment $\zeta$ of $[x, y]$ which has $C'$--bounded Hausdorff distance from $\gamma_j^k[\T, A_j - \T]$, implying that the length of $\zeta$ (and hence the length of $[x, y]$) is at least $A_j - 2T_0 -2C' > 2C'$, a contradiction to $d(x, y)\leq 2C'$. 
    \end{proof}

\begin{proof}[Proof of Proposition~\ref{prop:pentagon}]
        We show that the $\eta_i$ constructed above satisfy the desired properties. Let $\lambda$ be a geodesic starting at $x_0'$ and ending in $G\cdot x_0'$. If they exists, let $i_0, i_1$ be the unique integers such that
        \begin{align*}
            \diam{\pi_{\eta_{i_0}}(\lambda_{i_0})}\geq 4A_{i_0} + D_{B_{i_0}},\\
            \diam{\pi_{\eta_{i_1}}(\lambda)}\geq 4A_{i_1} + D_{B_{i_1}}.
        \end{align*}
        Lemma~\ref{fig:large-projection} shows that those integers are indeed unique. Let $i\neq j$ be integers such that $i, j, i_0, i_1$ are distinct. We have that 
        \begin{align*}
            p = \lambda_i\ast[x, y]\ast \lambda_j^{-1}\ast \eta_j^{-1}
        \end{align*}
        is a path connecting the endpoints of $\eta_i$. By Lemma~\ref{lem:strange-projection} we have that 
        \begin{align*}
            \pi_{\eta_i}(\lambda_j)\leq 4A_i + D_{B_i}.
        \end{align*}
        By Lemma~\ref{lem:small-pairwise-projection-of-etas} we have that
        \[
        \diam(\pi_{\eta_i}(\eta_j)) \leq 4A_i + D_{B_i}.
        \]
        Finally, by Lemma~\ref{lem:at-most-one-large-projection} we have that
        \begin{align*}
            \diam{\pi_{\eta_i}(\lambda_i)}\leq 4A_i + D_{B_i}.
        \end{align*}
        Consequently, we have that 
        \begin{align*}
            \diam{\pi_{\eta_i}([x, y])}\geq \abs{\eta_i} - 3(4A_i + D_{B_i} ) > 4A_i + D_{B_i}.
        \end{align*}
        Lemma~\ref{lem:morse-projects-short} implies that $[x, y]$ cannot be $N$--Morse. The proof that $[x', y']$ cannot be $N$--Morse is analogous.
\end{proof}

\subsection*{Proof of the theorem}

We are now ready to use Proposition~\ref{prop:pentagon} to prove Theorem~\ref{thm:vanishing_measure}. 

\begin{figure}
    \centering
    \includegraphics[width=\linewidth]{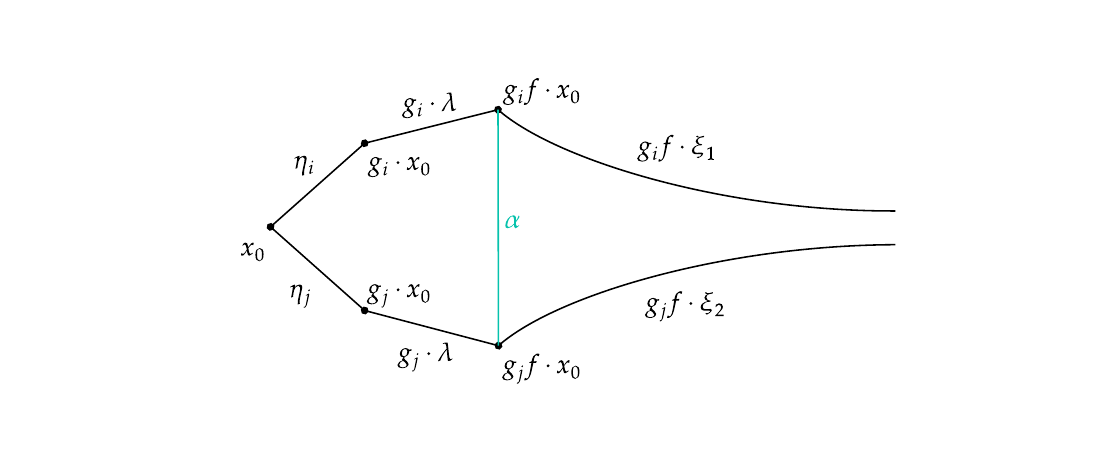}
    \caption{If $g_i f \cdot\xi_1$ and $g_j f \cdot\xi_2$ were translates of geodesic rays from the same Morse stratum, then $\alpha$ would be Morse. But we constructed $g_i$ to guarantee that does not happen.}
    \label{fig:pentagon}
\end{figure}

\begin{proof}[Proof of Theorem \ref{thm:vanishing_measure}] 
Let $N''$ be a Morse gauge and let $\mc{M}_{N''}(X)$ be the set of $N''$--Morse geodesic rays starting at $x_0$. Our goal is to apply Lemma~\ref{lem:-Tiozzo} to the set $Y = \rho(\mc{M}_{N''} (X))$. The condition on $\rho$ that $ \rho(\mc{M}_{N''} (X))$ has to be compact ensures that it is measurable.

Since $G$ has contraction, there exists a geodesic metric space $X'$ with basepoint $x_0'$ satisfying the Morse-local-to-global property and containing a strongly contracting geodesic ray $\gamma:[0, \infty) \to X'$. Fix a quasi-isometry $\phi : X \to X'$ which sends $x_0$ to $x_0'$. By \cite[Lemma~2.9]{C:Morse}, $\phi$ induces a map $\tilde{\phi} : \mc M(X)\to \mb X'$ and there exists a Morse gauge $N'$ such that $\tilde{\phi}(\mc M_{N''}(X)) \subset \partial_{x_0'}^{N'}X'$. Since $G$ acts geometrically on both $X$ and $X'$ we can assume that $\phi$ was chosen in a way such that $\tilde{\phi}$ is $G$--equivariant. Lastly, by Lemma \ref{new-monster}\ref{monster:triangle}, there exists a Morse gauge $N$ depending on $N'$ such if a triangle has two sides that are $N'$--Morse geodesic rays, the third side is $N$--Morse. 

Let $(\eta_i)_{i\in \N}$ be a sequence of geodesics as in Proposition~\ref{prop:pentagon} constructed in the space $X'$ with respect to the Morse gauge $N$. Let $g_i\in G$ be an element such that $g_i\cdot x_0'$ is the endpoint of $\eta_i$. Since the support of $\mu$ generates $G$ as a semigroup, for each $g_i$, there exists $m_i$ such that $\mu_{m_i}(g_i)>0$. Up to reordering, we may assume that 
\[
\mu_{m_i}(g_i) \geq \mu_{m_{i+1}}(g_{i+1}).
\]
Define the sequence $\epsilon_n := \mu_{m_{n+2}}(g_{n+2})$. Fix $f\in G$, and let $\lambda = [x_0',f\cdot x_0']$ and let $i_0, i_1$ be the forbidden indices of Proposition~\ref{prop:pentagon} for $\lambda$, and let $(g_{n_k})_{k\in \N}$ be the subsequence obtained from $(g_i)_{i\in \N}$ removing the elements $g_{i_0}$ and $g_{i_1}$. Observe that $i\leq n_i\leq i+2$ and hence 
\[
\mu_{m_{n_i}}(g_{n_i})\geq \mu_{m_{i+2}}(g_{i+2}) = \eps_{i},
\] 
which is a condition we need to be able to apply Lemma~\ref{lem:-Tiozzo}.

The only thing left in order to apply Lemma~\ref{lem:-Tiozzo} is to show that for any $g_i\neq  g_j \in (g_{n_k})_{k\in \N}$, 
\[
g_i f\rho(\mc{M}_{N''} (X)) \cap g_j f\rho(\mc{M}_{N''} (X)) = \emptyset
\qquad \text{and} \qquad
g_i f\rho(\mc{M}_{N''} (X)) \cap f\rho(\mc{M}_{N''} (X)) = \emptyset.
\]

Since $\rho$ is $G$--equivariant and maps rays that represent different elements in the Morse boundary to different points in $Z$, it suffices to show that for any pair of rays $\xi_1, \xi_2\in  \mc{M}_{N''} (X)$, 
\begin{itemize}
    \item the rays $ g_i f \xi_1$ and $ g_{j} f \xi_2$ must represent different points in the Morse boundary,
    \item the rays $ g_i f \xi_1$ and $f \xi_2$ must represent different points in the Morse boundary.
\end{itemize}

To show the first bullet point, note that if $[ g_i f \xi_1] =  [g_{j} f \xi_2]$, then $g_i f \tilde{\phi}([\xi_1]) = g_j f \tilde{\phi}([\xi_2])$. Let $\eta_i =[x_0',g_i \cdot x_0']$ and $\eta_j =[x_0',g_j \cdot x_0']$. Consider a geodesic segment $\alpha = [g_if\cdot x_0',g_jf\cdot x_0']$, this is depicted in Figure \ref{fig:pentagon}. By Proposition \ref{prop:pentagon}, $\alpha$ is not $N$--Morse, but if $g_i f \tilde{\phi}([\xi_1]) = g_j f \tilde{\phi}([\xi_2])$, $\alpha$ must be $N$--Morse by Lemma \ref{new-monster}\ref{monster:triangle}, a contradiction. The proof for the second bullet point is the same, replacing $g_j$ by the identity. Thus, Lemma~\ref{lem:-Tiozzo} yields that for any Morse gauge $N''$, we have $\nu( \rho(\mc{M}_{N''}(X))) = 0$. Since the Morse boundary is $\sigma$-compact by Theorem~\ref{thmintro:MLTG_implies_sigma_cpt} and the measure $\nu$ is subadditive, we conclude $\nu(\rho(\mc{M}(X))) =0$.
\end{proof}
\bibliography{mybib}
\bibliographystyle{alpha}

\end{document}